\documentclass[10pt]{amsart}
\numberwithin{equation}{section}
\usepackage{amsmath, latexsym, mathtools}
\usepackage{amssymb,amsthm}
\usepackage{amscd,graphics}

\usepackage{enumerate,mathrsfs,url}
\usepackage[usenames,dvipsnames]{color}
\usepackage[normalem]{ulem}

\DeclarePairedDelimiter\abs{\lvert}{\rvert}%

\newcommand{\real}{{\mathbb R}}
\newcommand{\complex}{{\mathbb C}}
\newcommand{\integer}{{\mathbb Z}}
\newcommand{\Natural}{{\mathbb N}}

\newcommand{\LL}{\mathcal{L}}

\newcommand{\UU}{\mathcal{U}}
\newcommand{\VV}{\mathcal{V}}
\newcommand{\NT}{{\overset{NT}{\rightarrow}}}

\newcommand{\Psiout}{\Psi_{\varphi,c_-}}
\newcommand{\rrlc}{rrl-continuation}
\newcommand{\gS}{\mathscr S}
\newcommand{\eps}{\epsilon}

\newcommand{\ph}{\varphi}          

\newcommand{\de}{\delta}

\newcommand{\ga}{\gamma}

\newcommand{\la}{\lambda}
\newcommand{\La}{\Lambda}
\newcommand{\sig}{\sigma}
\renewcommand{\th}{\theta}

\newcommand{\om}{\omega}

\newcommand{\col}{\colon}          

\newcommand{\pa}{\partial}

\newcommand{\ii}{^{-1}}

\newcommand{\ti}{\tilde}

\newcommand{\rhs}{{right-hand side}}

\newtheorem{proposition}{Proposition}[section]
\newtheorem{lemma}[proposition]{Lemma}

\newtheorem{theorem}[proposition]{Theorem}

\newtheorem{maintheorem}{Theorem}
\newtheorem{corollary}{Corollary}

\theoremstyle{remark}

\newtheorem{remark}[proposition]{Remark}

\theoremstyle{definition}

\begin{document}

\author{V. Baladi, S. Marmi, and D. Sauzin}
\title[Natural boundary for the susceptibility function]
{Natural boundary for the susceptibility function of generic piecewise expanding unimodal
maps}
\date{\today }

\address{V.~Baladi: D.M.A., UMR 8553, \'Ecole Normale Sup\'erieure,  75005 Paris, France.
Current address: University of Copenhagen, Department of Mathematical
Sciences, Universitetsparken 5, 2100 Copenhagen \o, Denmark}
\email{baladi@math.ku.dk}

\address{S.~Marmi: Scuola Normale Superiore, CNRS--UMI 3483 Fibonacci,
56126 Pisa, Italy}
\email{s.marmi@sns.it}

\address{D.~Sauzin: CNRS--UMI 3483 Fibonacci, Scuola Normale Superiore, 
56126 Pisa, Italy}
\email{david.sauzin@sns.it}
\begin{abstract} 
For  a piecewise expanding unimodal interval map $f$ with unique acim
$\mu$, a perturbation $X$, and an observable $\varphi$,
the susceptibility function is $\Psi_\varphi(z)=
\sum_{k=0}^\infty z^k  \int  X(x)   \varphi'( f^k)(x) (f^k)'(x) \, d\mu$. 
 Combining  previous results \cite{Ba, BS1}
(deduced from spectral properties
of Ruelle transfer operators) with recent work of Breuer--Simon
\cite{BrSi} (based on techniques from the spectral theory
of Jacobi matrices and a classical paper of Agmon \cite{Ag}), we show that  density of the postcritical
orbit (a generic condition) implies that  $\Psi_\varphi(z)$ has a strong natural boundary
on the unit circle. The Breuer--Simon method provides uncountably many candidates for the outer functions
of $\Psi_\varphi(z)$, associated to precritical orbits.
If the perturbation $X$ is horizontal,  a generic
condition (Birkhoff typicality of the postcritical orbit) implies
that the nontangential limit of $\Psi_\varphi(z)$ as $z\to 1$ 
exists and coincides with the  derivative of the acim with
respect to the map
(``linear response formula"). 
Applying the Wiener--Wintner theorem, we study
the singularity type of nontangential limits
of $\Psi_\varphi(z)$ as $z\to e^{i\omega}$ for real $\omega$.
An additional LIL typicality assumption on the postcritical orbit gives stronger results.

\end{abstract}

\thanks{ We thank I. Assani, H. Bruin, J. Buzzi,
J.-R. Chazottes,  P. Collet,   J.-P. Kahane, I. Melbourne, T. Persson,  
H. Queff\'elec,
D. Schnellmann, W.X. Shen, B. Simon,
and D. Smania,  
 for useful comments and references. 
 This work was started during a visit of VB to the Scuola Normale
Superiore in Pisa, and essentially completed during a visit of VB to
the  UMI Fibonacci of CNRS at Scuola Normale
Superiore in Pisa, the hospitality  of which is gratefully acknowledged.
The authors acknowledge the support of the Centro di Ricerca Matematica Ennio de
Giorgi.
The research leading to these results has received funding from the European
Comunity's Seventh Framework Program (FP7/2007--2013) under Grant Agreement
n.~236346. VB is a member of the ANR grant DynNonHyp, BLAN08-2\_313375.}
\maketitle

\section{Statement of results}

\subsection{Standing assumptions and notations}

Our standing assumptions on the dynamics $f$
are as follows: $I=[a,b]$ is a compact interval and $f:I\to I$ 
is a  piecewise expanding $C^3$ unimodal map. That is, $f$ is continuous on $I$ with $f(a)=f(b)=a$,
and there exists $c\in (a,b)$ (the {\em critical point}) so that
$f$ is $C^3$ with  $f'>1$ on $[a,c]$,
and $f$ is $C^3$ with $f'<-1$ on $[c,b]$.
We put $c_k=f^k(c)$ for $k\ge 0$.
The sequence $\{ c_k\}_{k \ge 1}$ is called
the {\em postcritical orbit.}  (The interval $[c_2,c_1]$ is forward-invariant.) We set
$$\lambda=\inf_{x\ne c} |f'(x)| >1 \, .
$$
We shall assume that the critical point is not periodic (this ensures that
$f$ is ``good" in the sense of \cite{BS1}). In fact,
except in Remark~\ref{finite}, we shall always assume 
that the postcritical orbit
$\{ c_k \}_{k \ge 1}$ is infinite. 
We let $\mu=\rho\, dx$ be the unique absolutely continuous invariant
probability measure of $f$.  
Recall that $(f, \mu)$ is always ergodic, and
\cite[Prop. 3.6]{Babook} that $(f, \mu)$ is mixing if
$f$ is topologically mixing on $[c_2, c_1]$. Except in Remark ~\ref{nonmix} we assume throughout that $f$
is topologically mixing on $[c_2, c_1]$.

Let 
$X\in C^2(f(I))$ satisfy $X(a)=0$. 
We say that $X$ is 
{\em horizontal} for $f$ if
\begin{align}\label{13}
\sum_{n=0}^{\infty} \frac{ X(c_{n+1})}{(f^n)'(c_1)} =0\, .
\end{align}
Consider a $C^2$ family $t\mapsto f_t$  of $C^3$ piecewise expanding interval
maps (defined for $t \in (-1,1)$, say), with $f_0=f$,
so that all $f_t$ have the same critical point $c$ and satisfy $f_t(a)=f_t(b)=a$,
and  so that $\partial_t f_t|_{t=0}= X \circ f$. (In other words,
$X$ represents a perturbation of the dynamics.)
We denote by $\mu_t$ the unique  absolutely continuous invariant
probability measure of $f_t$.

The main motivation for the present work is provided by the following question:
What is the smoothness of the map $t \mapsto \int \varphi\, d\mu_t$
(for suitable functions $\varphi:I \to \complex$)?  In connection with this
question, Ruelle \cite{Ru1, Ru2} proposed to study the {\em susceptibility
function,} i.e., the formal power series associated
to $f$, $X$ and $\varphi \in C^1(I)$ by
\begin{equation}
\label{7}\Psi_\varphi(z)=
\sum_{k=0}^\infty z^k  \int  X(x) \,(\ph'\circ f^k)(x)(f^k)'(x) \, \rho(x) dx  \, .
\end{equation}
(See the survey \cite{BaN}.)
We can assume without loss of generality that $\int \varphi\, d\mu=0$.
Integrating each term by parts and using that 
$\rho$ is of bounded variation, the definition 
of $\Psi_\varphi(z)$ is extended to 
$\varphi \in C^0(I)$ (see \eqref{formm} for an explicit formula).

For generic $X$, the following function 
encodes the singular behaviour
of $\Psi_\varphi$:
\begin{equation}	\label{eqdefsigph}
\sigma_\varphi(z)= \sum_{k=0}^\infty \varphi(c_{k+1}) z^k \, .
\end{equation}
Indeed, the following theorem is essentially proved in \cite{Ba, BS1} (see Appendix~\ref{gapp} for  details, including explicit
formulas for $\UU$, $\VV_\varphi$, and $\Psi^{hol}_\varphi$):

\begin{theorem}[Susceptibility function and linear response]\label{oldtheorem}
If $\varphi \in C^0(I)$ 
and $\int \varphi\, d\mu=0$, the susceptibility function~$\Psi_\varphi$ is
holomorphic in the open unit disc, and
\begin{align}\label{formula}
\Psi_\varphi(z) =
\Psi^{sing}_\varphi(z)+  \Psi^{hol}_\varphi(z)
= \UU(z) \sigma_\varphi (z) + \VV_\varphi(z) + \Psi^{hol}_\varphi(z)\, ,
\end{align}
where  the functions $\UU$ and~$\VV_\varphi$ are holomorphic in 
$|z| > \lambda^{-1}$, 
there exists $\kappa <1$ so that $\Psi^{hol}_\varphi$ is holomorphic in
$|z|<\kappa^{-1}$, and the function~$\UU$ depends only on $f$ and $X$. In addition,
$X$ is horizontal if and only if  $\UU(1)=0$. 

If $X$ is horizontal, then for any $C^2$ family $t\mapsto
f_t$ of $C^3$ piecewise expanding interval maps so that $f_0=f$ and $\partial_t
f_t|_{t=0}= X \circ f$, we have
\begin{equation}\label{mystery}
 \pa_t \Big( \int \ph \, d\mu_t \Big) |_{t=0}=\VV_\varphi(1) + \Psi^{hol}_\varphi(1)
 \, .
\end{equation}
\end{theorem}

Our aim in this article is to investigate the behaviour of the
functions~$\sig_\ph(z)$ and~$\Psi_\ph(z)$ near the unit circle and particularly near
the point~$z=1$. 

In view of~\eqref{eqdefsigph}, the distribution of the points $c_1,c_2,\ldots$  in~$I$ will obviously play a role.
The sequence $\{c_k\}_{k\ge 1}$ is generically  dense in the interval $[c_2,c_1]$.
In fact,
Schnellmann \cite[Thm 6.1, comments after Cor.\ 2]{Sc} showed that for generic%
\footnote{I.e., almost all maps in a transversal
family, we refer to \cite{Sc} for a precise statement.}
 piecewise expanding $C^2$ maps $f$, the critical point
$c$ is {\em Birkhoff typical,} that is, for all continuous functions $\varphi :I \to \complex$
we have
$$
\lim_{m \to \infty} \frac{1}{m}\sum_{k=0}^{m-1} \varphi(c_{k+1})= \int \varphi \, d\mu\, .
$$
 (If $c$ is Birkhoff typical, its orbit is dense in $[c_2,c_1]$, because our mixing
 assumption implies that  $\rho$ is bounded from below on its
 support $[c_2,c_1]$.)

Our main results are the following:

Theorem~\ref{natural} in Subsection ~\ref{1.2} gives a strong natural boundary for
$\Psi_\varphi$, and 
Theorem \ref{thmpoly} in Subsection ~\ref{1.3} guarantees ``uncountably
many renascent right limits," 
both are
proved in Section~\ref{22} under the assumption that the critical orbit is dense.  

Under the assumption that the critical orbit is Birkhoff typical, Theorem
\ref{abelianNTzero} and Corollary~\ref{corlinresp} in Subsection ~\ref{1.4}
give nonpolarity of nontangential limits of $\Psi_\varphi(z)$ at $z=e^{i\omega}$, and, assuming horizontality, the tangential limit at
$z=1$  (proofs in Section~\ref{33}), while Theorem ~\ref{derr} in  Subsection~\ref{1.6} gives
nontangential limits of derivatives of $\Psi_\varphi(z)$ at $z=1$  (proof in Section~\ref{55}). 

Then, Theorem~
\ref{abelianloglog} in Section~\ref{1.5} refines Theorem
\ref{abelianNTzero}, under a (possibly generic) ``iterated logarithm law" condition
on the postcritical orbit (proof in Section~\ref{33}).
Finally, Section ~\ref{44} is devoted to  observables $\varphi$ which are coboundaries, the
main result there is Proposition~\ref{coboundary} (assuming again only Birkhoff typicality of the periodic orbit).

\begin{remark}[Finite postcritical orbits]\label{finite}
The (highly non generic) case when the postcritical orbit is finite is much
simpler \cite[\S 5]{Ba}. Then the susceptibility function is holomorphic in the open
unit disc and meromorphic in a disc of radius $\kappa^{-1}>1$. Its
possible poles in that disc are roots of unity, and $z=1$ is
not a pole if $X$ is horizontal.
The function~$\sig_\ph$ itself is rational:
$\sig_\ph(z) = P(z) + \frac{Q(z)}{1-z^p}$,
where $p\ge 1$ is the primitive period of the postcritical periodic point
$f^p(c_m)=c_m$, and
$$P(z) = \sum_{k=1}^m \ph(c_k) z^{k-1}\, , \quad 
Q(z) = \sum_{k=m}^{m+p-1} \ph(c_k) z^{k-1}\, .
$$
The  residue of $\sigma_\varphi$ at $\omega$ with $e^{ip \omega}=1$ is equal to
$\frac{1}{p}\sum_{k=m}^{m+p-1} e^{i\omega(k-1)}  \varphi(c_k)$.
(Note that if  there is a continuous 
$\psi$ so that $\varphi=e^{i\omega}\psi -\psi \circ f$,
 then this residue vanishes.)
\end{remark}

\begin{remark}[Non-mixing maps]\label{nonmix}
If the postcritical orbit is dense in $[c_2, c_1]$ and $f$ is not topologically mixing 
on $[c_2,c_1]$, then 
there exist two intervals $J_1$ and $J_2$ with $J_1\cap J_2=\{y_0\}$,
$f(y_0)=y_0$, and $J_1\cup J_2=[c_2,c_1]$, with $f(J_1)=J_2$, $f(J_2)=J_1$, and
$f^2|_{J_i}$ topologically mixing for $i=1,2$ (see e.g.\ \cite[\S VI.5, Thm ~46 and remark thereafter]{BC}, 
noting that the dense postcritical orbit assumption ensures that $f$ is transitive on
$[c_2, c_1]$). The interested reader can  exploit the present remark to study  non-mixing maps $f$
for which the postcritical
orbit is dense (angles $\omega$ with $e^{i\omega}=-1$ must be treated separately).
\end{remark}

\subsection{Strong natural boundary}\label{1.2}

Following \cite{BrSi}, a function $g(z)$ holomorphic in the disc $|z|<1$ is
said to have a {\em strong natural boundary} on the unit circle if,
for every nonempty interval $(\om_1,\om_2)$,

\begin{equation}	\label{eqdefstrgnatb}
\sup_{0<r<1} \int_{\om_1}^{\om_2} |g(r e^{i\omega})|\, d\omega =\infty\, .
\end{equation}

If $g$ has a strong natural boundary
then $g$ has an
$L^\infty$ natural boundary, that is, $|g|$ is unbounded in every sector
$\{r e^{i\omega}\mid r \in (0,1),\,
\om\in (\om_1,\om_2)\}$.

Using recent deep results of Breuer and Simon \cite[Thm 3.1]{BrSi} based on previous work
of Agmon \cite{Ag}, we prove in Section~\ref{22}:

\begin{maintheorem}[Strong natural boundary]\label{natural}
Assume that  the postcritical orbit is dense
in $[c_2, c_1]$.
Then, for any continuous $\varphi$ which is not constant on $[c_2,c_1]$ and so
that $\int \varphi\, d\mu=0$, the unit circle is a strong natural boundary
for~$\sig_\ph$ and for the susceptibility function $\Psi_\varphi$.
\end{maintheorem}

For a sequence $\{a_k\}_{k=0}^\infty$ in a topological space~$E$, 
a {\em right limit} \cite{BrSi} is any two-sided sequence $\{ b_n\}_{n=-\infty}^\infty$ of~$E$
for which there exists an increasing sequence of positive integers 
$\{k_j\}_{j=0}^\infty$ such that
$\lim_{j\to\infty} a_{n+k_j} = b_n$ for every $n\in\integer$.
Theorem~3.1 in \cite{BrSi}  reads as follows:
\begin{quote}
Suppose that a bounded sequence of complex numbers $\{a_k\}_{k=0}^\infty$ has two distinct right limits
$\{ b_n\}_{n=-\infty}^\infty$ and $\{\tilde b_n\}_{n=-\infty}^\infty$, and that 
there exists $N\in\integer$ such that either $b_n = \tilde b_n$ for all $n\ge N$,
or $b_n = \tilde b_n$ for all $n\le N$, then 
the unit circle is a strong natural boundary for the power series $\sum_{k\ge 0} a_k z^k$.
\end{quote}

The proof of Theorem~\ref{natural} will consist in exhibiting, in the
case of the sequence $\ph(c_{k+1})$, distinct right limits which coincide for
$n\ge 0$, entailing a strong natural boundary for~$\sig_\ph$ and thus
for~$\Psi_\ph$.

\subsection{Renascent right limits}\label{1.3}

In fact, we shall see in Theorem~\ref{thmpoly} that the set of right limits of
the sequence $\{\varphi(c_{k+1})\}_{k \ge 0}$ is multifarious well beyond the requirement of
\cite[Thm 3.1]{BrSi}. To state this result, we introduce some
terminology:

A right limit $\{b_n\}_{n =-\infty}^\infty$ of a bounded sequence of complex
numbers $\{ a_k \}_{k=0}^\infty$ is called {\em renascent} if $b_n=a_n$ for all $n\ge 0$,
and {\em weakly renascent} if there is a finite set of integers $F$ so
that $b_n=a_n$ for all $n \in \Natural\setminus F$. 
If $\{b_n\}_{n=-\infty}^\infty$ is renascent then the function $g_{b_-}(z) = - \sum_{n=-1}^{-\infty} b_n
z^n$, which is holomorphic in $|z|>1$ and vanishes at~$\infty$, is called
an {\em rrl-continuation} of the function
$g(z) = \sum_{k=0}^\infty a_k z^k$

Following \cite{BrSi}, we say that a bounded sequence $\{ b_n\}_{n \in
\integer}$ is {\em reflectionless} on an arc $J = \{\, e^{i\om} \mid \om \in
(\om_1,\om_2) \,\}$ of the unit circle if the function $g_+(z) =
\sum_{n=0}^\infty b_n z^n$ admits an analytic continuation in a neighbourhood
of~$J$ and the value of this continuation at any~$z$ with $|z|>1$ is $g_-(z) =
-\sum_{n=-1}^{-\infty} b_n z^n$.

Let $g(z)=\sum_{k=0}^\infty a_k z^k$ with $\sup|a_k|<\infty$
and $J = \{\, e^{i\om} \mid \om \in (\om_1,\om_2) \,\}$.
According to Theorem~1.4 of \cite{BrSi},
if \eqref{eqdefstrgnatb} is violated
(e.g.\ if $g$ admits an analytic extension through~$J$),
then  all right limits  $\{ b_n\}_{n=-\infty}^\infty$ of $\{a_k\}_{k=0}^\infty$ are reflectionless on~$J$.
In particular, if $\{a_k\}_{k=0}^\infty$ admits a renascent right limit 
$\{b_n\}_{n=-\infty}^\infty$,
then $g_+=g$ by definition. Thus, if there exists an arc~$J$ in the
neighbourhood of which~$g$ admits an analytic continuation, this continuation
must coincide with each \rrlc\ $g_-=g_{b_-}$ outside the closed unit disc.
In such a case, all renascent right limits coincide and all the analytic
continuations of~$g$ through arcs of the unit circle match.
(Mutatis mutandis, weakly renascent right limits enjoy similar properties.)
If on the contrary there is no analytic continuation for~$g$ across any arc
of the unit circle, then the unit circle is a strong natural boundary (because
the renascent right limit is not reflectionless on any arc).
However, we still may think of the \rrlc s $g_{b_-}$ as being somewhat
``connected'' to~$g$, as in the theories of generalised analytic continuation \cite{RS}
or of monogenic continuation \cite{Bor}, \cite{MS}.%
\footnote{In a nutshell, Borel monogenic functions are a generalisation of analytic functions of one complex variable, which allows functions to be defined on closed sets which may even have empty interior. They share most of the properties of analytic functions---especially Cauchy's integral formula---including, in some cases relevant for our scopes, being a quasianalytic space. }
We refer to Appendix~\ref{exxx} for examples.

We require more terminology:
A power series $\sum_{k =0}^\infty a_k z^k$ 
with $\sup|a_k|< \infty$ 
 is called {\em polygenous} if the set  of  renascent
right limits of the sequence $\{a_k\}_{k=0}^\infty$  contains at least two elements.

We do not know any polygenous power series admitting a Borel monogenic
extension beyond the unit circle, and we are tempted to guess that  such examples do not
exist.
An even stronger conjecture would be that, whenever a power series has a monogenic extension
beyond the unit circle and an \rrlc, they necessarily coincide.

By \cite[Thm 3.1]{BrSi}, the unit
circle is a strong natural boundary for any polygenous series. The following
result, which is proved in Section~\ref{22}, is thus a reinforcement of
Theorem~\ref{natural}:
We say that a function $\varphi$ is
{\em $f$-symmetric on $[c_2,c_1]$} if $\varphi(x)= \varphi(y)$ for any
$x$, $y$ in $[c_2,c_1]$ so that $f(x)=f(y)\in [c_2,c_1]$.

\begin{maintheorem}[Uncountably many renascent right limits for the postcritical orbit]\label{thmpoly}
Assume that the postcritical orbit is dense in $[c_2,c_1]$.
\begin{enumerate}[(i)]
\item	\label{itemRLc}
The right limits of the sequence $\{c_{k+1}\}_{k\ge 0}$ are exactly the
complete orbits of~$f$ contained in $[c_2,c_1]$, i.e., the two-sided sequences
$\{x_n\}_{n=-\infty}^\infty$ of $[c_2,c_1]$ such that
$x_{n+1} = f(x_n)$ for all $n\in\integer$.
\item \label{itemRRLph}
For any continuous $\varphi$ which is not $f$-symmetric on $[c_2,c_1]$,
the set  obtained by identifying renascent right limits of
$\{ \varphi(c_{k+1})\}_{k \ge 0}$ 
which differ only on a finite set is uncountable.
\end{enumerate}
\end{maintheorem}

We say that $\{ y_n\}_{n \le -1}$ is a {\em precritical orbit} if 
$y_n\in[c_2,c_1]$ with $f(y_{n-1})=y_{n}$ for all $n\le -1$,
and ${y_{-1}=c}$. 
Slightly abusing notation, we write $c_-=\{ c_{n+1}\}_{n \le -1}$ for a precritical
orbit. 
Theorem~\ref{thmpoly}\eqref{itemRLc} implies that
the renascent right limits of the sequence
$\{c_{k+1}\}_{k\ge0}$ are the complete orbits of~$c_1$ contained in
$[c_2,c_1]$, i.e., the two-sided sequences obtained by glueing any
precritical orbit with $\{c_{k+1}\}_{k=0}^\infty$.
The argument for  Theorem~ \ref{thmpoly}\eqref{itemRRLph} uses the fact (proved there) 
that there are uncountably many precritical orbits.

\begin{remark}[Renascent right limits and formal resolvents in $|z|>1$]\label{formalres}
By Theorem~\ref{thmpoly},  the  rrl-con\-ti\-nua\-tions  of $\sig_\ph$ are of the form
\begin{equation}\label{nosurprise}
\sigma_{\varphi,c_-}(z)=-\sum_{n \le -1} \varphi(c_{n+1}) z^n \, ,
\end{equation}
where $c_-=\{c_{n+1}\}_{n\le-1}$ is any precritical orbit.
Clearly, for $|z|<1$
\begin{equation}\label{Koop}
\sigma_\varphi(z)=\sum_{n =0}^{\infty} [(zU)^{n} \varphi](c_1) =[ (1-z U)^{-1} \varphi ] (c_1)\, ,
\end{equation}
where $U(\varphi) = \varphi \circ f$ is the Koopman operator acting
(e.g.) on $C^0$, which has spectral  radius equal to $1$, but is not
invertible. It is therefore hardly surprising that the \rrlc s,
which we expect to be candidates for the ``outer function" of $\sigma_\varphi$
outside of the unit disc, are nothing else than
\begin{equation}\label{nosurprise'}
\sigma_{\varphi,c_-}(z)=-\sum_{n =1}^{\infty} [(zU)^{-n} \varphi](c_1) \, ,
\end{equation}
where the operator sequence $U^{-n}$ is any sequence satisfying
$U^{-n} U= U^{-n+1}$, so that (\ref{nosurprise'}) can formally be written as
$$
-[(1- (zU)^{-1})^{-1} (zU)^{-1} \varphi](c_1) \, .
$$
(This remark should be put in parallel with the manipulations in \cite[\S 17]{Ru3}.)
Note  that \eqref{cobtrick} in Proposition~\ref{coboundary} below
is  not surprising either in view of the present discussion, since
$(1-e^{i\omega}U)\psi (c)=\varphi(c)$ for the function $\varphi$
considered there, although it is not clear how to exploit (\ref{Koop}) to prove ~\eqref{cobtrick}
without using Wiener--Wintner.
\end{remark}

\subsection{Nontangential limits --- The outer functions $\sigma_{\varphi,c_-}$ and $\Psiout$}\label{1.4}

We move on to results on the convergence of $\Psi_\varphi(z)$ when $z$ tends
non-tangentially in the open unit disc (i.e., without leaving a fixed open
sector based at $e^{i\omega}$ and contained inside the open unit disc) to some
$e^{i\omega}$ with $\omega\in \real$, denoted $z\NT e^{i\omega}$.
(Clearly,
if $\lim_{z\NT e^{i\omega}} \psi(z)$ exists then the Abelian limit
$\lim_{r \to 1} \psi(r e^{i\omega})$ exists.)
The following theorem is proved in Section~\ref{33} using
the Wiener--Wintner results from Appendix~\ref{BWW}:

\begin{maintheorem}[Nonpolar nontangential limits] \label{abelianNTzero}
Assume that the critical point  is
Birkhoff typical.
Let $\varphi \in C^0(I)$.
For any $\omega \in \real$ with $e^{i\omega} \ne 1$, we have
\begin{equation}\label{notan}
\lim_{z\NT e^{i\omega}} (z-e^{i\omega}) \sigma_\varphi(z)=0\, ,
\mbox{ and thus }
\lim_{z\NT e^{i\omega}} (z-e^{i\omega}) \Psi_\varphi(z)= 0 \, .
\end{equation}
If $\int \varphi\, d\mu=0$ then the above also holds for $e^{i\omega}=1$.
\end{maintheorem}

In the horizontal case, we get a linear response interpretation
of the nontangential limit to $z=1$ by applying the results of \cite{BS1} recalled in Theorem~\ref{oldtheorem}:

\begin{corollary}[Linear response and nontangential limit at~$1$]	\label{corlinresp}
Assume that the critical point  is
Birkhoff typical, $X$ is horizontal, and
let
$\varphi \in C^0(I)$ with $\int \varphi\, d\mu=0$.
Then 
\begin{equation}	\label{eqabelPsiph}
\lim_{z\NT 1} \Psi_\varphi(z)= \VV_\varphi(1) + \Psi^{hol}_\varphi(1)\, .
\end{equation}
In addition, for any $C^2$ family $t\mapsto f_t$ of $C^3$ piecewise expanding
interval maps so that $f_0=f$ and $\partial_t f_t|_{t=0}= X \circ f$,  we have 
\begin{equation}\label{wasknown}
\lim_{z\NT 1} \Psi_\varphi(z) =\pa_t \Big( \int \ph \, d\mu_t \Big) |_{t=0} \, .
\end{equation}
\end{corollary}

Recalling \eqref{mystery}, if $X$ is horizontal, 
then $\lim_{z\NT 1}\Psi_\varphi(z)-\Psi^{hol}_\varphi(1)=\VV_\varphi(1)$ 
(see also the
convergent resummation~\eqref{mystery2} for $\VV_\varphi(1)$).

Replacing nontangential limit by Abelian limit,  the identity
\eqref{wasknown} was proved in
\cite{BS1}, under a stronger assumption on $X$, see Subsection~\ref{1.6}.

\begin{proof}[Proof of Corollary~\ref{corlinresp}]
By Theorem~\ref{oldtheorem}, we have $\UU(1)=0$. Thus, \eqref{notan} from
Theorem~\ref{abelianNTzero} for $e^{i\omega}=1$
gives $\lim_{z\NT 1} \UU(z) \sigma_\varphi(z)=0$, which yields~\eqref{eqabelPsiph}.
Then use \eqref{mystery} from Theorem~\ref{oldtheorem}.
\end{proof}

If the sequence $c_-=\{c_{n+1}\}_{n\le-1}$ is a precritical orbit then $\{ \varphi
(c_{n+1})\}_{n \in \integer}$ is a renascent right limit for $\{ \varphi
(c_{k+1})\}_{k \ge 0}$ giving rise to the \rrlc\ $\sigma_{\varphi,c_-}$
defined in \eqref{nosurprise}.  
If $\ph$ is not $f$-symmetric on $[c_2,c_1]$, there are uncountably many such
\rrlc s by
Theorem~\ref{thmpoly}(\ref{itemRRLph}). To each renascent right limit, we associate a candidate
for the ``outer" susceptibility function by  setting
\begin{equation}
\Psiout(z)=\UU(z) \sigma_{\varphi,c_-}(z)+ \VV_\varphi(z)+\Psi^{hol}_\varphi(z)\, .
\end{equation}
Note that $\sigma_{\varphi,c_-}(z)$ is holomorphic in $|z|>1$ (it vanishes at infinity), and
$\Psiout(z)$ is holomorphic in $1<|z|< \kappa^{-1}$.

We   say that a {\em precritical orbit $\{c_{n+1}\}_{n\le -1}$ is Birkhoff typical} if
for all $\varphi \in C^0$ we have
$$
\lim_{m \to \infty}
\frac{1}{m} \sum_{n=-1}^{-m} \varphi(c_{n+1}) = \int \varphi \, d \mu \, .
$$

Theorem~\ref{abelianNTzero} and  Corollary~\ref{corlinresp} have an analogue for $\sigma_{\varphi,c_-}$ and  $\Psiout$
(see Section~\ref{33}):

\begin{theorem}[Nontangential limits for outer functions]\label{noproofneeded}
Let $c_-=\{c_{n+1}\}_{n\le -1}$ be a  Birk\-hoff typical precritical orbit.

For any $\varphi \in C^0(I)$ and any $\omega \in \real$ with $e^{i\omega} \ne 1$, we have
\begin{equation*}
\lim_{z\NT e^{i\omega}} (z^{-1}-e^{-i\omega}) \sigma_{\varphi,c_-}(z^{-1})=0 \, 
\mbox{and thus } \lim_{z\NT e^{i\omega}} (z^{-1}-e^{-i\omega}) \Psiout(z^{-1})= 0 .
\end{equation*}
If $\int \varphi\, d\mu=0$ then the above also holds for $e^{i\omega}=1$.

Let
$\varphi \in C^0(I)$ with $\int \varphi\, d\mu=0$,
and let $X$ be horizontal. 
Then 
\begin{equation} \label{Psioutt}	
\lim_{z\NT 1} \Psiout(z^{-1})=\lim_{z\NT 1} \Psi_\varphi(z)= \VV_\varphi(1) + \Psi^{hol}_\varphi(1)\, .
\end{equation}
\end{theorem}

\begin{remark}[Birkhoff typical precritical orbits] A precritical orbit can accumulate at a repelling
periodic orbit, in which case it is not Birkhoff typical.
But we expect that generic
piecewise expanding maps have infinitely many 
Birkhoff typical precritical orbits.  Note that if there
 were
a unique Birkhoff typical  precritical orbit for $f$, then this would single out
a renascent limit giving an \rrlc\ more natural than all others. 
See also Remark~\ref{p10}.
\end{remark}


\subsection{Higher order horizontality and derivatives}\label{1.6}

In \cite[Prop. 4.6]{BS1} we gave a sufficient condition  for  existence of the Abelian limit at $z=1$ of the susceptibility
function in the piecewise expanding case.
We asked in \cite{BaN} whether this condition was necessary.  
By Corollary~\ref{corlinresp}, the answer to this question is negative
(horizontality with Birkhoff typicality of the postcritical orbit suffices). 
However the sufficient additional condition from \cite[Prop 4.6]{BS1}
\begin{align}\label{ordertwo}
\sum_{n=1}^{\infty} \frac{ nX(c_{n+1})}{(f^n)'(c_1)} =0\, ,
\end{align}
appears in a more natural
manner in  Theorem~\ref{derr} below.
We say that $X$ is horizontal of order $H\ge 1$
if
\begin{align}\label{orderH}
\sum_{n=\ell}^{\infty} \frac{n!} {(n-\ell)!}\frac{X(c_{n+1})}{(f^n)'(c_1)} =0\, ,
\quad \forall 0\le \ell \le H-1 \, .
\end{align}
The following result is proved in
Section~\ref{55}:

\begin{maintheorem}[Non tangential limits of derivatives of $\Psi$ and $\Psiout$ at $1$]\label{derr}
Assume that the critical point is Birkhoff typical. Let
$\varphi \in C^0(I)$ satisfy $\int \varphi\, d\mu=0$.  
Assume that $X$ is horizontal of order $H$ for some $H\ge 1$.

Then 
we have, for each $0\le \ell \le H-1$,
$$\lim_{z\NT 1}(\UU(z) \sigma_\varphi)^{(\ell)}(z)=0
\mbox{ and thus }
\lim_{z\NT 1}\Psi^{(\ell)}_\varphi(z)
=\VV_\varphi^{(\ell)}(1)+ (\Psi_\varphi ^{hol})^{(\ell)}(1)\, .
$$
In addition, for any Birkhoff  typical precritical orbit $c_-=\{c_{n+1}\}_{n \le -1}$, 
we have 
$$\lim_{z\NT 1}(\UU(z) \sigma_{\varphi,c_-})^{(\ell)}(z^{-1})=0\, ,
\quad \forall 0 \le \ell \le H-1\, ,
$$
so that 
$$\lim_{z\NT 1}(\Psiout)^{(\ell)}(z^{-1})=\lim_{z\NT 1}\Psi^{(\ell)}_\varphi(z)
=\VV_\varphi^{(\ell)}(1)+ (\Psi_\varphi ^{hol})^{(\ell)}(1)\, , \quad \forall 0 \le \ell \le H-1\, .
$$
\end{maintheorem}

\begin{remark} [Elusiveness of quasi-analytic extensions]
Horizontality of order $\ell$ means that
$\UU^{(k)}(1)=0$ for $k=0, \dots, \ell-1$.
Since $\UU$ is holomorphic close to $z=1$, horizontality of all orders would mean
that $\UU\equiv 0$. 
In addition,  since $X$ is continuous, 
\eqref{defU} below implies that $\UU \equiv 0$  
is  equivalent to $X \equiv 0$ if the postcritical
orbit is dense. So there
is no hope of proving the existence of a non holomorphic quasi-analytic extension
at $z=1$ by using the horizontality mechanism. 
\end{remark}

\subsection{Law of the iterated logarithm}\label{1.5}

In order to get more precise results, we shall make assumptions involving the
(rotated) law of the iterated logarithm (LIL) for the postcritical orbit.
Let us first recall known results. 
We say that $\varphi$ is a coboundary if there exists  $\psi \in C^0$ so that
$\varphi=\psi -\psi \circ f$.
The ordinary LIL for a piecewise expanding
map $f$ says
(see \cite{HK}, who prove an almost sure invariance principle from which the LIL
follows, {see also \cite[\S6]{Bro}, and} \cite{Go} and references therein)
that for any $\varphi$ {of bounded variation} 
which is not a coboundary 
there exists $C(\varphi)>0$ so that for Lebesgue amost every $x$
\begin{equation}\label{limsup}
\limsup_{m \to \infty}\biggl | \frac{\sum_{k=1}^m \varphi(f^k(x))}
{\sqrt {m \log \log m}} \biggr | = C(\varphi)\, .
\end{equation}
It is not unreasonable to expect that  the postcritical
orbit is typical for the LIL, i.e., that we may take $x=c$  in (\ref{limsup}).  
(Indeed, several experts \cite{Be,Sh} expect that the postcritical orbit may be
generically typical for the LIL in the setting of smooth unimodal maps.)

Let now $\omega\in \real$ be so that $e^{i\omega}\ne 1$.
We say that $\varphi$ is
an $\omega$-coboundary if there exists
$\psi\in C^0$ so that $\varphi=\psi -e^{i\omega} \psi \circ f$.
We refer to \cite{Wi} for an $\omega$-rotated LIL
in the probabilistic setting. In a deterministic dynamical setting, the
rotated LIL says that for every real $\omega$ and any H\"older%
\footnote{Very recent results \cite{CC} on the CLT for rotated ergodic sums
indicate that the smoothness assumption on $\varphi$ may be unnecessary if one is willing to restrict to an $f$-dependent full measure set of $\omega$'s.}
$\varphi$, which is not
an $\omega$-coboundary,
there exists $C(\varphi, \omega)>0$ so that for Lebesgue almost every $x$
\begin{equation}\label{rLIL}
\limsup_{m \to \infty} \biggl | \frac{\sum_{k=1}^m e^{i k\omega} \varphi(f^k(x))}
{\sqrt {m \log \log m}} \biggr | = C(\varphi,\omega)\, .
\end{equation}
 To obtain the above,  first put together the results 
of \cite{FMT}  on mixing skew products 
$(x, \theta)\mapsto (f(x), \theta+\omega)$
and observables $\Phi(\theta, x)=e^{i\theta} \varphi(x)$
for H\"older $\varphi$ to get the almost sure invariance
principle and thus the LIL for a.e. $(\theta, x)$. 
Then, use \cite{MNJLM} to remove the mixing
assumption, allowing rational $\omega$.
Finally, apply \cite{MNSD} to get the LIL
for almost every $x$ while fixing $\theta=0$. See
also   \cite[Remark 2.5]{MNSD} for Young towers (which include
our piecewise expanding maps) and \cite[\S 5.b]{NMA} for the nondegeneracy-coboundary 
condition. (We are grateful to I. Melbourne for explaining his
results.)

The above discussion  gives
hope that the following typicality condition  holds for 
a large (at least countable
dense?) set of  $\omega$s
if  $f$ is generic and $\varphi$ is smooth:

The critical point is {\em typical for the upper $\omega$-rotated LIL}
and $\varphi\in C^0$ if
$$
\limsup_{m \to \infty} \biggl | \frac{\sum_{k=1}^m e^{i k\omega} \varphi(c_k)} 
{\sqrt {m \log \log m}} \biggr | < \infty\, .
$$
(Observe that we may as well write ``$\sup$'' instead of ``$\limsup$.'')

The following theorem is proved in Section~\ref{33}:
\begin{maintheorem}[Consequences of upper $\omega$-rotated LIL
typicality of $c$] \label{abelianloglog}
Let $\varphi \in C^0(I)$ and $\omega \in \real$ with $e^{i\omega}\ne 1$.
Assume that the critical point is typical
for the upper $\omega$-rotated LIL and $\varphi$.
Then, for any open sector $\gS$ contained in
$\{ \la\ii < |z| < 1 \}$
with vertex at~$e^{i\om}$, there exists $C(\gS)>0$ such that
\begin{equation}\label{notanLIL}
\max \{ |\sig_\ph(z)|, \, |\Psi_\ph(z)|\} \le
\frac{C(\gS)}{ |z-e^{i\om}|^{1/2}} \Big(\log\log\frac{1}{|z-e^{i\om}|}\Big)^{1/2} \, ,
\qquad \forall z\in\gS \, .
\end{equation}
If in addition $\int \varphi\, d\mu=0$ then the above also holds for $e^{i\omega}=1$.
\end{maintheorem}

The above result may indicate that, generically, $\Psi_\varphi$ could 
have uncountably  many singularities on the unit circle  which are ``slightly worse" than ramifications 
of order two. If true, this fact could be related to the polygenous property
of $\{\varphi(c_{k+1})\}_{k\ge 0}$. See also Remark~\ref{LIL} about potential lower bounds.

\smallskip

Theorem~\ref{abelianloglog} has an obvious analogue for the outer functions $\sigma_{\varphi,c_-}$ and  $\Psiout$
associated to precritical orbits which are typical for an upper rotated LIL. We refrain from making
a formal statement.

\begin{remark}[Rotated-LIL typical precritical orbits] \label{p10}
Remark~3.6~b) in
\cite{MNJLM} says the following: ``Passing to the natural extension \cite{Ro},
it follows from the methods in \cite{FMT} that the law of the iterated logarithm
(and much more, including the almost sure invariance principle) can be proved in
backwards time." This gives the LIL for H\"older observables and some backwards orbit of
almost every $x$.  It is yet unknown whether  the critical point
itself satisfies the LIL, but H. Bruin \cite{Br}
has heuristical arguments which  encourage us to expect that,   generically,
some precritical orbit is
typical for the LIL in the piecewise expanding setting.
\end{remark}

\subsection{Open questions}

Here are three open questions which arise
from our results (see  also Remark~\ref{LIL}): 

\begin{itemize}
\item
If $X$ is not horizontal and $\varphi$ is not a coboundary, does the
Abelian (or nontangential) limit of $\Psi(z)$
as $z\to 1$ ever exist? (See also \eqref{cobb}.)%
\footnote{Clearly, the series evaluated at $z=1$ is in general divergent.
In \cite[Prop. 4.5]{BS1} we proved that the resummation there diverges for
some $C^\infty$ function  $\varphi$.}
\item
If $X$ is horizontal but not horizontal of order two,
does the
Abelian (or nontangential) limit  of $\Psi'(z)$ as $z\to 1$ ever
exist?
\item
For $\omega\ne 0$, if $\varphi$ is not an  $\omega$-coboundary,
does the
Abelian (or nontangential) limit as $z\to e^{i\omega}$ ever converge to a finite number? 
(See also \eqref{cobb}.) Does the answer depend on the diophantine properties
of $\omega$?
\end{itemize}

We end this section by some comments on the smooth unimodal case.
The piecewise expanding situation considered here can be viewed as
a toy model for the more difficult smooth unimodal case
\cite{Ru2, Ru3, BS2}.  In the postcritically finite (Misiurewicz-Thurston)
smooth unimodal
case, the susceptibility
function admits a meromorphic extension in a disc of radius larger
than $1$. In some Misiurewicz--Thurston cases, Ruelle \cite{Ru2} was able to
prove that $z=1$ is not a pole without assuming horizontality.
However,  there are examples  \cite{BBS} of Misiurewicz--Thurston parameters
where  (Whitney)
linear response is violated  when horizontality does not hold. 
Existence of an extension which is holomorphic
at $z=1$  thus does
not guarantee linear response.

 When the postcritical orbit of a smooth unimodal
map is infinite  and slowly recurrent (Collet-Eckmann, topologically
slowly recurrent, or polynomially recurrent),  it is expected
\cite{Ru3}
that the natural boundary will be a circle of radius strictly smaller than $1$,
while the derivative of the SRB measure (in the horizontal case)
should be related to a suitable extension%
\footnote{Perhaps only along the real axis? The nature
of this extension remains admittedly mysterious.} of $\Psi_\ph$ evaluated at $z=1$,
at least under generic assumptions (as in Corollary~\ref{corlinresp}). 
In the analytic Misiurewicz unimodal case
an  analogue $\tilde \sigma_\varphi$ of $\sigma_\varphi$ can be obtained from 
\cite[end of \S 17, \S 16(b)]{Ru3}. 
Horizontality would perhaps guarantee that the contribution
$\widetilde \UU(z) \tilde \sigma_{\varphi,c_-}(z)$ of  any outer function corresponding
to a renascent right limit continuation $\tilde \sigma_{\varphi,c_-}(z)$
of $\tilde \sigma_\varphi(z)$ would vanish at $z=1$ (in the same way as horizontality implied
$\UU(1)=0$ in the present setting). 
In the Misiurewicz--Thurston case, the above mentioned result of
Ruelle \cite{Ru2} implies that the singular term $\widetilde \UU(z)\tilde \sigma_\varphi$
is  meromorphic with no pole at $z=1$, but this  neither implies that 
this singular term vanishes at $z=1$, nor that $\Psi_\varphi(z)$ is
real analytic on $[0,1]$.  Summarising,
although we are perhaps closer to understanding
the misleading behaviour \cite{Ru2} of finite postcritical orbits, the Borel monogenic extension  hoped
for in \cite{BaN}
remains elusive.


\section{Strong natural boundary and renascent right limits}\label{22}

In this section, we prove Theorems~\ref{natural} and~\ref{thmpoly}.
We begin with a simple lemma about right limits (which is implicit in \cite{BrSi}):

\begin{lemma}\label{diagonal}
Let $\{a_k\}_{k=0}^\infty$ be a sequence in a compact metric space~$E$.  
Then any increasing sequence of positive integers $\{m_j\}_{j=0}^\infty$ admits a 
subsequence $\{k_j\}_{j=0}^\infty$ such that, for each $n\in\integer$,
the limit 
$\lim_{j\to\infty} a_{n+k_j}$
exists.

In particular, for any accumulation point~$b$ of $\{a_k\}_{k=0}^\infty$, 
there exists a right limit $\{ b_n\}_{n=-\infty}^\infty$ with $b_0=b$.
\end{lemma}

\begin{proof}
Pick an arbitrary point~$e$ in~$E$, and define for each $k\ge0$ a two-sided sequence 
$\vec a^{[k]} \in E^\integer$ by
$a^{[k]}_n = a_{n+k}$ if $n \ge -k$ and
$e$ if $n < -k$.
By Tikhonov's theorem, $E^\integer$ is compact for the product topology, 
hence the sequence $(\vec a^{[m_j]})_{j\ge0}$ admits a subsequence
$(\vec a^{[k_j]})_{j\ge0}$ which converges to a limit~$\vec b$ in $E^\integer$,
which exactly means $\lim_{j \to \infty} a_{n+k_j}= b_n$ for each $n\in\integer$.

The last statement follows by choosing $\{m_j\}_{j=0}^\infty$ so that 
$\lim_{j \to \infty} a_{m_j} = b$, and applying the first part of the lemma.
\end{proof}


To prove Theorem~\ref{natural}, we shall also need a 
well-known consequence of our mixing assumption:

\begin{lemma}	\label{lempastepsdense}
For any $x_0 \in (c_2,c_1)$ and any $\eps>0$, there exists a positive integer $\ell$ so
that the set $f^{-\ell}({x_0})$ is $\eps$-dense in $I=[a,b]$.
\end{lemma}

\begin{proof}
Since $f$ is topologically mixing on $[c_2,c_1]$, it has (see%
\footnote{
See also  \cite[Lemma 2]{Ko}
and \cite{Ke78} for earlier proofs of the ``weak covering''  property, which
by \cite[Thm 4.4]{Li} implies covering.} the reference to
Hofbauer in \cite[Prop.~ 2.6 and Appendix B]{Bu}, as explained in
\cite[p.\ 641]{vBC})
the following covering property \cite{Li}: For every $n \ge 1$ there is 
$\ell(n) <\infty$ such that, for every interval of monotonicity $I_n$ of $f^n$, 
the set $f^{\ell(n)} (I_n)$ covers $[c_2,c_1]$ up to finitely many points (note that
Hofbauer did not assume continuity of $f$). 
Since $f^{\ell(n)}$ is continuous, $f^{\ell(n)} (I_n)$ is an interval and
hence contains~${x_0}$.
Therefore, $f^{-\ell(n)}({x_0})\cap I_n\ne \emptyset$ for each interval of monotonicity
of $f^n$.  It suffices to take $n(\epsilon)$ so that each interval of
monotonicity of $f^n$ has length $< \epsilon$ (this is possible since each such
interval has length $\le |b-a|\lambda^{-n}$).
\end{proof}


\begin{proof}[Proof of Theorem~\ref{natural}]
Let $y,\ti y \in [c_2,c_1]$ be such that $\ph(y)\neq \ph(\ti y)$.
Since $\ph$ is continuous, we can take $\de>0$ so that
\[
\text{$x\in J$ and $\ti x\in\ti J$} \quad \Longrightarrow \quad
\ph(x)\neq\ph(\ti x)\, ,
\]
where $J = [y-\de,y+\de]\cap [c_2,c_1]$ and $\ti J = [\ti y-\de,\ti y+\de]\cap [c_2,c_1]$
are non-trivial intervals. 
Since $f$ is mixing, by Lemma~\ref{lempastepsdense} applied to $x_0=c$, we can find
$\ell\ge 1$ such that $f^{-\ell}(c)\cap J$ and $f^{-\ell}(c)\cap \ti J$ are
non-empty.

Pick $x \in f^{-\ell}(c)\cap J$ and $\ti x \in f^{-\ell}(c)\cap \ti J$.
Since $\{c_{k}\}_{k\ge 0}$ is dense in $[c_2,c_1]$, we can find two integer
sequences $\{m_j\}$ and $\{\ti m_j\}$ such that
$x = \lim c_{m_j+1}$ and $\ti x = \lim c_{\ti m_j+1}$.
By Lemma~\ref{diagonal}, we get two right limits $\{ x_n \}_{n\in\integer}$ and
$\{ \ti x_n \}_{n\in\integer}$ for the sequence $\{c_{k+1}\}_{k\ge 0}$ such that
$x_0 = x$ and $\ti x_0 = \ti x$.
By construction, $f^\ell(x_0) = f^\ell(\ti x_0) = c$ but $\ph(x_0)\neq\ph(\ti x_0)$.

Now the relation $c_{k+2} = f(c_{k+1})$ and the continuity of~$f$ imply that 
$x_{n+1} = f(x_n)$ and $\ti x_{n+1} = f(\ti x_n)$ for every $n\in\integer$, hence 
\[
n \ge \ell \quad\Longrightarrow\quad x_n = \ti x_n\, .
\]
By continuity of~$\ph$, the two-sided sequences $\{b_n\}_{n\in\integer}$ and
$\{\ti b_n\}_{n\in\integer}$ defined by
\[
b_n = \ph(x_n)\, , \quad \ti b_n = \ph(\ti x_n), \qquad n\in\integer
\]
are right limits for the sequence $\{\ph(c_{k+1})\}_{k\ge 0}$. They are distinct, since
$b_0\neq\ti b_0$, but they coincide for $n\ge\ell$. 
We can thus apply \cite[Thm 3.1]{BrSi}, which ensures that
$\sig_\ph(z) = \sum_{k \ge 0} \varphi(c_{k+1}) z^k$ has a strong natural boundary
on $\{|z|=1\}$. 
By Theorem~\ref{oldtheorem}, this must be the case for~$\Psi_\ph(z)$ too.
\end{proof}

\begin{remark}
Our argument hinges on the noninvertibility of~$f$. 
Breuer--Simon \cite[Thm 7.1]{BrSi} have a similar result for homeomorphisms
of the circle with dense orbits, but for observables with discontinuities: The
discontinuities there achieve the effect we obtain here from the
noninvertibility.
\end{remark}

Let us move on to the proof of Theorem~\ref{thmpoly}. 
We shall need the following lemma:

\begin{lemma}	\label{lemvardiag}
Let $(E,d)$ be a metric space and $T \col E \to E$ be a continuous map.
Suppose that $\ga\in E$ has a dense orbit $\{\ga_k=T^k(\ga)\}_{k\ge 0}$.
Then the right limits of this sequence are exactly the
complete orbits of~$T$, i.e.,  the two-sided sequences
$\{x_n\}_{n=-\infty}^\infty$ of $E$ such that
$x_{n+1} = T(x_n)$ for all $n\in\integer$.
\end{lemma}

\begin{proof}
We first suppose that $\{x_n\}_{n\in\integer}$ is a right limit.
Let $\{k_j\}$ be an integer sequence such that $\lim_{j\to\infty} \ga_{k_j+n} =
x_n$ for every $n\in\integer$.
Then, for any $n\in\integer$, the limit~$x_{n+1}$ of the sequence $\ga_{k_j+n+1} =
T(\ga_{k_j+n})$ must coincide with $T(x_n)$ (by continuity of~$T$),
hence $\{x_n\}_{n\in\integer}$ is a complete orbit of~$T$.

Suppose now that $\{x_n\}_{n\in\integer}$ is any complete orbit of~$T$.
For $m,j\ge 0$, we denote by $B_{m,j}$ the open ball of centre~$x_{-m}$ and
radius $\frac{1}{j+1}$.
For each $j\ge 0$, the set
\[
V_j = \bigcap_{m=0}^j T^{-m}(B_{j-m,j}) 
= \bigcap_{m=0}^j T^{-(j-m)}(B_{m,j}) 
\]
is open and contains~$x_{-j}$ (because $T^m(x_{-j}) = x_{-(j-m)}$ for
all~$m$), in particular it is non-empty and must contain at least one point
of our dense sequence $\{\ga_k\}$.
We can thus define an integer sequence $\{k_j\}$ by
\[
k_j = j + \min\{ k\ge 0 \mid T^k(\ga) \in V_j \}, \qquad j\ge 0.
\]
This sequence tends to~$\infty$ (since $k_j\ge j$) and, for any fixed $m\ge0$, 
\begin{multline*}
j\ge m
\quad\Longrightarrow\quad 
T^{k_j-j}(\ga) \in V_j \subset T^{-(j-m)}(B_{m,j}) \\
\quad\Longrightarrow\quad 
T^{k_j-m}(\ga) = T^{j-m}\big( T^{k_j-j}(\ga) \big) \in B_{m,j}\, ,
\end{multline*}
hence $\lim_{j\to\infty} \ga_{k_j-m} = x_{-m}$.
In particular, $\lim \ga_{k_j} = x_0$ and, by continuity of~$T$, 
$\lim_{j\to\infty} \ga_{k_j+n} = x_{n}$ for every $n\ge1$.
We thus have proved that $\{x_n\}_{n=-\infty}^\infty$ is a right limit of $\{\ga_k\}_{k=1}^\infty$.
\end{proof}


\begin{proof}[Proof of Theorem~\ref{thmpoly}]
We can restrict $f$ to $[c_2,c_1]$, because this interval is forward-invariant and contains all the
elements of any right limit of $\{c_{k+1}\}$ (indeed, all of them are
accumulation points of this orbit, which is contained in~$[c_2,c_1]$).
Statement \eqref{itemRLc} thus follows from Lemma~\ref{lemvardiag}.

For \eqref{itemRRLph}, we consider $\varphi(x)=x$ as a warmup
case. Then it suffices to show that the set of complete orbits of $c_1$
contained in $[c_2,c_1]$ is uncountable. This is an immediate consequence of
the fact that
there are uncountably many precritical orbits. This fact can be easily proved as
follows:
If $y_n\in[c_2,c_1]$ is given with $n\le-1$ and $f^{-n}(y_n)={c_1}$,
then $y_n < c_1$ (because we assumed that the postscritical orbit is not
finite), thus it has two distinct preimages;
there are one or two possibilities for $y_{n-1}\in[c_2,c_1]$ 
(according as $y_n<c_3$ or $y_n\ge c_3$), 
but there are always at least two possibilities for
$y_{n-2}\in[c_2,c_1]$. 
(The reader is invited to draw a picture.)

Let us now consider~$\ph$ continuous and not $f$-symmetric on
$[c_2,c_1]$. Let $x\ne y$ so that $f(x)=f(y)=v$ be a pair where symmetry is
violated, i.e., $\varphi(x)\ne \varphi(y)$. We may assume that $\min(x,y)> c_3$
and $\max(x,y)<c_1$.  Then, since $\varphi$ is continuous, there exist $\delta
>0$ so that, for any $\tilde x\ne \tilde y$ with $f(\tilde x)=f(\tilde y)$
and $|f(\tilde x)-v|< \delta$, we have $\varphi(\tilde x)\ne \varphi(\tilde y)$.
To prove claim \eqref{itemRRLph}, it suffices to show that there are
uncountably many precritical orbits so that $|c_n -v|<\delta$ for infinitely
many $n \le -1$. For this,  since $f$ is mixing, we may apply
Lemma~\ref{lempastepsdense} for $\epsilon < \delta/2$ (say) iteratively, first
for $x_0=c$, and then, infinitely many times, to all $x_0$ in the
$\epsilon$-dense set from the previous iteration which satisfy $|x_0- v|<\delta$
(there are at least two different such $x_0$ at each step since $\epsilon < \delta/2$).
\end{proof}

\section{Nonpolar nontangential limits as $z\to e^{i\omega}$} \label{33}


In this section we prove Theorems~\ref{abelianNTzero} and~\ref{abelianloglog}.
(Theorem~\ref{noproofneeded} is proved exactly like Theorem~\ref{abelianNTzero},
replacing $z^k$ by $z^{-k}$, the iterate $f^k$ by the inverse branch corresponding to the
chosen precritical orbit, and noting that the proof of Lemma~\ref{WW} shows that,
if the precritical orbit is Birkhoff typical, then it is Wiener--Wintner typical.)

\begin{proof}[Proof of Theorem~\ref{abelianNTzero}]
For $\omega \in \real$, it is enough to consider $
\sig_\ph(e^{i\om} z)$ for $z$ in  an open sector~$\gS$ contained in 
$\{ \la\ii < |z| < 1 \}$
with vertex at~$1$.
We introduce a notation for the partial sums:
\begin{equation}\label{partt}
S_k(e^{i\om}) =
\sum_{j=0}^{k-1} e^{ij\om} \ph(c_{j+1}), \qquad k\ge1\, .
\end{equation}
Then, by Abel summation,
\begin{equation}\label{abb}
\sig_\ph(e^{i\om} z) = (1-z) \sum_{k=1}^\infty S_k(e^{i\om}) z^{k-1}
\quad \text{for $|z|<1$.}
\end{equation}
Let $\om\in\real$ (if $e^{i\omega}=1$ we assume
$\int \varphi\, d\mu=0$). Since the  critical point of~$f$ is Birkhoff typical,
Lemma~\ref{WW} (Wiener--Wintner) gives $S_k(e^{i\om}) = o(k)$.
On the other hand, $|S_k(e^{i\om})|\le k \sup|\ph|$.
Therefore, for each $\eps >0$, there is $K$ so that
$|S_k(e^{i\om})|\le \eps k$ for all $k \ge K$ and
\begin{align*}
\bigl | \sum_{k=1}^\infty S_k(e^{i\om}) z^{k-1}\bigr | 
&\le \sum_{k=1}^{K-1} |S_k(e^{i\om})| + \sum_{k=K}^\infty |S_k(e^{i\om})| |z|^{k-1} \\
&\le K^2 |\sup \varphi| + \eps \big(1-|z|\big)^{-2}
\quad \text{for $|z|<1$.}
\end{align*}
Since there exists a constant $C(\gS)>0$ such that
\begin{equation}	\label{ineqgS}
z\in\gS \enspace\Longrightarrow\enspace
1-|z| \le |1-z| \le C(\gS) \big( 1-|z| \big)\, ,
\end{equation}
we obtain from \eqref{abb} that
\[
|(1-z) \sig_\ph(e^{i\om} z)| \le 
|1-z|^2 K^2 |\sup \varphi| +
C(\gS)^2 \eps\, ,
\]
for any $z\in\gS$, whence the conclusion follows.
\end{proof}

\begin{remark}
Our argument for proving Theorem~\ref{abelianNTzero}
is similar to the argument in \cite[pp. 12--13]{KuNi} mentioned in
Remark~\ref{Heckex} about Hecke's example. 
(The difference is that, in \cite{KuNi}, a nonzero Fourier coefficient caused a
polar-type blowup, while here the Wiener--Wintner theorem on the contrary
ensures convergence to zero.)
\end{remark}

\begin{proof}[Proof of Theorem~\ref{abelianloglog}]
We now assume that the critical point of~$f$ is typical for the upper
$\omega$-rotated LIL and $\varphi$,
i.e., $|S_k(e^{i\om})| \le D \sqrt{k\log\log k}$ for all $k\ge3$, with a certain
$D = D(e^{i\om}) >0$.
By~\eqref{abb}, we get
\begin{gather}
\label{ineqsigL}
|\sig_\ph(e^{i\om} z)| \le |1-z| \Big( 3\sup|\ph| + D L\big(|z|\big) \Big)
\quad \text{for $|z|<1$,} \\
\intertext{with}
L(r) = \sum_{k=3}^\infty r^{k-1} \sqrt{k\log\log k} 
\quad \text{for $0\le r < 1$.}
\end{gather}

Assume for a moment that
for any $s_0 \in \big( 0, \log\frac{1}{1-e\ii} \big)$ there exists $M>0$ such that
\begin{equation}	\label{ineqL}
L(r) \le M (1-r)^{-3/2} \Big( \log\log \frac{1}{1-r} \Big)^{1/2}
\quad \text{for $e^{-s_0} \le r <1$.}
\end{equation}
Then, putting together~\eqref{ineqsigL} and~\eqref{ineqL}, we get
\begin{equation}
|\sig_\ph(e^{i\om} z)| \le \widetilde M(\om) |1-z| \big(1-|z|\big)^{-3/2}
\Big( \log\log \frac{1}{1-|z|} \Big)^{1/2}
\quad \text{for $\la\ii \le |z| <1$,}
\end{equation}
from which inequality~\eqref{notanLIL} for~$\sig_\ph$ follows by~\eqref{ineqgS}.
Inequality~\eqref{notanLIL} for~$\Psi_\ph$ then follows by Theorem~\ref{oldtheorem}.

To complete the proof of Theorem~\ref{abelianloglog}, it remains to show \eqref{ineqL}.
The condition $r \ge e^{-s_0}$ ensures
$\log\log \frac{1}{1-r} \ge \log\log \frac{1}{1-e^{-s_0}} > 0$,
so that the \rhs\ of~\eqref{ineqL} is bounded from below. It is sufficient to consider
arbitrarily small $s_0$.
For any $s\ge0$, we introduce the notation
\begin{equation*}
A_s(x) = e^{-sx} \sqrt{x \log\log x}
\quad \text{for $x\ge3$}.
\end{equation*}
Let us fix $r = e^{-s}$ with $0<s\le s_0$, so that
\[
rL(r) = \sum_{k=3}^\infty A_0(k) r^k 
= \sum_{k=3}^\infty A_s(k)\, .
\]
We observe that the function~$A_0$ is increasing on $[3,+\infty)$,
while~$A_s$ is decreasing on $[s\ii,+\infty)$ as soon as $s_0\le e^{-e}$
(because the logarithmic derivative of~$A_s$ is 
$-s + \frac{1}{2x}\big( 1 + \frac{1}{\log x \log\log x} \big)
< - s + \frac{1+e\ii}{2x} < 0$),
hence
\[
rL(r) \le \sum_{3\le k < s\ii+1} A_0(k) + \sum_{k\ge s\ii+1} A_s(k)
\le \int_3^{s\ii+2} A_0(x) \,dx + \int_{s\ii}^{+\infty} A_s(x)\, dx\, .
\]
Integrating by parts, we get
\[
\int_3^{s\ii+2} A_0(x) \,dx  = 
\Big[\tfrac{2}{3} x^{3/2} (\log\log x)^{1/2}\Big]_3^{s\ii+2} 
- \tfrac{2}{3}\int_3^{s\ii+2} x^{3/2} \big((\log\log x)^{1/2}\big)'\, dx\, ,
\]
which is 
$ < \tfrac{2}{3}({s\ii+2})^{3/2} (\log\log({s\ii+2}))^{1/2}
\le \text{const} \big(\frac{1}{s}\big)^{3/2} \big(\log\log\frac{1}{s}\big)^{1/2}$, 
while 
\begin{multline*}
B = \int_{s\ii}^{+\infty} A_s(x)\, dx = 
\tfrac{1}{s} \Big[ e^{-sx} A_0(x) \Big]_{+\infty}^{s\ii}
+ \int_{s\ii}^{+\infty} e^{-sx} \tfrac{A_0'(x)}{s}\, dx \\
< \tfrac{e\ii}{s} A_0\big(\tfrac{1}{s}\big) + 
\tfrac{1+e\ii}{2} \int_{s\ii}^{+\infty} \tfrac{1}{sx}{A_s(x)}\, dx 
< \tfrac{e\ii}{s} A_0\big(\tfrac{1}{s}\big) + \tfrac{1+e\ii}{2} B
\end{multline*}
(because the logarithmic derivative of~$A_0$ is 
$< \frac{1+e\ii}{2x}$ for $x \ge s\ii$).
Thus, $B < \tfrac{2}{e-1} \tfrac{1}{s} A_0\big(\tfrac{1}{s}\big)
= \text{const} \big(\frac{1}{s}\big)^{3/2} \big(\log\log\frac{1}{s}\big)^{1/2}$.
Since $1-r \sim s$, the proof is complete.
\end{proof}

\begin{remark}[Lower bounds from the law of the iterated logarithm]
\label{LIL}
Assume that the postcritical orbit is
typical for the ordinary LIL \eqref{limsup}. Then,
in view of the lower bound analogue of \eqref{ineqL}, if  $X$ is non horizontal, it is not absurd to guess that
$\sup_{r \in [0, 1)}| \Psi_\varphi(r)|=\infty$
when
$\varphi$ is not a coboundary. Since the LIL would only give  a sequence
of times $m_j \to \infty$ so that
$ \sum_{k=1}^{m_j} \varphi(c_k)  > C \sqrt{m_j \log\log m_j}$, it is not  clear how  to transform
this  intuition into a rigorous proof. Presumably, some information
on the sequence $m_j$ would help.

In view of studying $\omega\in \real$ with $e^{i\omega}\ne 1$, recall that
if $\varphi$ is continuous and $c$ is typical for the Birkhoff
theorem,  then for Lebesgue
almost every $\omega$ we have \cite{A} 
$$
\lim_m \int \bigl | 
\frac{1}{\sqrt m} \sum_{k=1}^m \varphi(c_k) e^{i k \omega}
- (\int |\varphi|^2) ^{1/2}\bigr |^2 =0 \, ,
$$
and thus  for Lebesgue almost every $\omega$,
\begin{equation}\label{lower}
\limsup_m \frac{1}{\sqrt m} 
\sum_{k=1}^m \varphi(c_k) \Re (e^{i k \omega}) \ge \int | \varphi|^2 \, d\mu \, .
\end{equation}
We would get (stronger) estimates for fixed $\omega$ by assuming that the critical
point is typical for the  $\omega$-rotated LIL (\ref{rLIL}) and
a non-$\omega$-coboundary $\varphi$.
This indicates that, 
for generic $f$ and $\varphi$, we may hope to get
$\sup_{r \in [0, 1)}| \Psi_\varphi(r e^{i\omega})|=\infty
$ for a large set of $\omega$s (countable dense?). 
Just like in the discussion for $e^{i\omega}=1$ above,  a rigorous proof 
would probably require information on the sequence $m_j \to \infty$ corresponding
to the $\limsup$ in \eqref{limsup}. Diophantine properties of $\omega$ could play
a role: Possibly  the nontangential limit could exist for such angles, leading potentially to 
continuous extension(s) of $\sigma_\varphi$ on a subset of positive measure of the circle.
\end{remark}

\section{Higher order horizontality and nontangential limits of derivatives}
\label{55}

\begin{proof}[Proof of Theorem~\ref{derr}]
We start with the claim about the derivatives of $\UU(z)\sigma_\varphi(z)$.
It suffices to consider $H \ge 2$.

Assume first that $H=2$.
By \eqref{derivative} and \eqref{defU} in Appendix~\ref{gapp}, if $X$ is horizontal of order two then $\UU(z)(z-1)^{-2}$
is holomorphic in $|z|> \lambda^{-1}$.
Therefore, to  get $\lim_{z\NT 1} (\UU(z) \sigma_\varphi)'(z)=0$, we only need to prove that
$$
((z-1)^2\sigma_\varphi(z))'=
2 (z-1) \sigma_\varphi(z)+ (z-1)^2  (\sigma_\varphi(z))'
$$
converges to zero when $z\NT 1$. The first term converges
to $0$ as $z \NT 1$ by Theorem~\ref{abelianNTzero}.
We shall next see that $(z-1)^2  \sigma_\varphi'(z)$ 
also converges to
$0$.

Recalling \eqref{partt}, the proof of Theorem~\ref{abelianNTzero} shows that 
\[
\sigma_\varphi(z)  = (1-z) S(z), \quad \text{with}\quad
S(z) = \sum_{k=1}^\infty S_k(1) z^k\, ,
\]
and that $\lim_{z \NT 1} (z- 1)^2 S(z) = 0$.
Therefore 
$$(z-1)^2  \sigma_\varphi'(z) = (1-z)^3 S'(z) - (z-1)^2  S(z)\, , $$ 
and we just need to check that
$\lim_{z \NT 1} (z- 1)^3 S'(z) = 0$.
Let $\epsilon>0$. Since $S_k (1)= o(k)$ (by Birkhoff typicality of the critical
point), we can choose $K\ge1$ such that
$|S_k(1)| \le \epsilon (k+1)$ for all $k \ge K$, hence
\begin{align*}
|S'(z)| = \Big| \sum_{k=1}^\infty k S_k(1) z^{k-1} \Big| &\le
\sum_{k=1}^K |k S_k(1) | + \epsilon \sum_{k=1}^\infty k(k+1) |z|^{k-1}\\
&\le K^3 \max|\varphi| + \epsilon (1 - |z|)^{-3}\, ,
\end{align*}
and the conclusion follows.

Let $H\ge 3$.  If $X$ is horizontal of  order   $H$, then \eqref{derivative'} implies that
$\UU(z)(z-1)^{-H}$
is holomorphic in $|z|> \lambda^{-1}$.
To  get $\lim_{z\NT 1} (\UU(z) \sigma_\varphi)^{(H-1)}(z)=0$, we only need to prove that
$
((z-1)^H\sigma_\varphi(z))^{(H-1)}
$
converges to zero when $z\NT 1$. The only term which needs
to be considered is $(z-1)^H  \sigma_\varphi^{(H-1)}(z)$. Proceeding inductively,
we reduce to showing
$$
\lim_{z \NT 1}(z-1)^{H+1}  S^{(H-1)}(z)=0 \, .
$$
Using $S^{(H-1)}(z)=\sum_{k=H-1}^\infty \frac{k!}{(k-(H-1))!}S_k(1) z^k$, we proceed similarly
as in the case $H=2$.

The results we just proved on $(\UU\sigma_\varphi)^{(\ell)}$ immediately give that
$\Psi^{(\ell)}_\varphi(z)$
converges as $z\NT 1$, for all $\ell\le H-1$.

\smallskip
It remains to consider the  \rrlc s associated to
a Birkhoff typical precritical orbit $\{c_{n+1}\}_{n \le -1}$: Replacing $z^k$ by $z^{-k}$ and $f^k$ by its
appropriate inverse branch, the previous arguments give for any $\ell\le H-1$
$$
\lim_{z \NT 1}
\bigl ( (z-1)^\ell (\sigma_{\varphi,c_-}(z))^{(\ell-1)} \bigr ) =0\, .
$$
\end{proof}

\section{Nontangential limits  as $z\to e^{i\omega}$ for $\omega$-coboundaries}
\label{44}

Given $\om\in\real$, we say that $\varphi$ is a {\em $C^0$ postcritical $\om$-coboundary}
if there exists $\psi \in C^0$ so that
\begin{equation}\label{cobb}
\varphi(c_k)= \psi(c_k) -e^{i\omega} (\psi \circ f)(c_k)=
 \psi(c_k) -e^{i\omega} \psi (c_{k+1})\, , \quad
 \forall k\ge 1\, .
\end{equation}
If the postcritical orbit is dense, this is equivalent to existence of
a continuous $\psi$ so that $\varphi=\psi - e^{i\omega} \psi \circ f$.
This degeneracy condition on $\varphi$ is well studied.
If $e^{i\omega}=1$, it implies
the vanishing of $\sum_{k=0}^{p -1}\varphi(f^k(x))$ along every
periodic orbit $f^p(x)=x$, giving a closed set of infinite codimension.
If $e^{i\omega}\ne 1$, see e.g.\ \cite[\S 5.b]{NMA}, where the degeneracy condition
is still satisfied  only on a closed set of infinite
codimension.

Our result below does not require horizontality, even for
$e^{i\omega}=1$:

\begin{proposition} [Nontangential limits for continuous $\omega$-coboundaries]
\label{coboundary}
Assume that  the critical point is Birkhoff typical.
Let $\om\in\real$.
If $\varphi \in C^0(I)$ is a $C^0$ postcritical $\omega$-coboundary for
$\psi$  then
\begin{equation}\label{cobtrick}
\lim _{z\NT e^{i\omega}}\sigma_\varphi(z)=\psi(c_1)\, , 
\end{equation}
and, therefore,
$$
\lim _{z\NT e^{i\omega}}\Psi_\varphi(z) =
\UU(e^{i\om}) \psi(c_1) 
+\VV_\varphi(e^{i\omega}) + \Psi^{hol}_\varphi(e^{i\omega})\, .
$$
\end{proposition}

For a precritical orbit
$\{c_{n+1}\}_{n \le -1}$, we say that  
$\varphi$ is a $C^0$ {\em precritical $\omega$-coboundary,} 
if there exists
$\psi \in C^0$ so that
\begin{equation*}
\varphi(c_n)= 
 \psi(c_n) -e^{i\omega} \psi (c_{n+1})\, , \quad
 \forall n\le 0\, .
\end{equation*}
If the precritical orbit is dense, this is equivalent to existence of
a continuous $\psi$ so that $\varphi=\psi - e^{i\omega} \psi \circ f$.
Adapting the proof of Proposition ~\ref{coboundary} below, one shows that if
the precritical orbit $c_-=\{c_{n+1}\}_{n\le -1}$ is Birkhoff typical (and thus, by an adaptation
of the proof of Lemma~\ref{WW}, Wiener--Wintner typical)  
and if $\varphi$ is a $C^0$  precritical $\omega$-coboundary,
then the corresponding   rrl-continuation 
$\sigma_{\varphi,c_-}$  of $\sigma_\varphi$
satisfies
$$
\lim_{ z\NT e^{i\omega}} \sigma_{\varphi,c_-}(z^{-1})=
\psi(c_1)=\lim_{z\NT e^{i\omega}} \sigma_\varphi(z) \, .
$$ 
For an $\omega$-coboundary $\varphi$ there is thus a (tenuous) link between $\sigma_\varphi$ and 
each rrl-continuation $\sigma_{\varphi,c_-}$ associated
to a Birkhoff typical precritical orbit $c_-$. This can be compared to the
identity $g_{b_-}(z)=g(z^{-1})+\frac{1}{2}$ in Example~\ref{Heckex}.

\begin{proof}[Proof of Proposition~\ref{coboundary}]
Since $c$ is Birkhoff typical it is Wiener--Wintner typical by Lemma~\ref{WW}.
In particular, recalling \eqref{partt},  we have
$S_k(e^{i\omega})=o(k)$ 
for all $\omega \in \real$ (where we used $\int \varphi\, d\mu=0$ if $e^{i\omega}=1$,  since $\varphi=\psi-\psi \circ f$).
We first show that this Wiener--Wintner type property, combined with the assumption that
$\varphi$ is a $C^0$ postcritical
$\omega$-coboundary, implies ~\eqref{cobtrick}:
The coboundary assumption yields  $\psi \in C^0$ with
$S_k(e^{i\omega})=\psi(c_1)-e^{i k\omega} \psi(c_{k+1})$
for all $k\ge 1$.
Therefore,  
\begin{align}
 (1-z) \sum_{k=1}^\infty S_k(e^{i\omega})   z^{k-1}&=
 \psi(c_1)- (1-z) \sum_{k=1}^\infty e^{i k\omega} \psi(c_{k+1})   z^{k-1} \\
&=
 \psi(c_1)- (1-z) \sigma_\psi (e^{i\omega} z) \, .
\end{align}
The Wiener--Wintner type property $S_k(e^{i\omega})=o(k)$ and the proof of
Theorem~\ref{abelianNTzero} (replacing $\varphi$ by $\psi$)
give 
$\lim_{z\NT 1}  (1-z) \sigma_\psi (e^{i\omega} z) =0$.
Therefore, recalling \eqref{abb}, we get 
$$
\lim_{z \NT e^{i\omega}} \sigma_\varphi(z)=\psi(c_1) \, .
$$
and we have proved ~\eqref{cobtrick}.

To conclude, combine Theorem~\ref{oldtheorem} with \eqref{cobtrick} and the proof of 
Theorem~\ref{abelianNTzero}.
\end{proof}


\appendix
\section{Examples of renascent right limits and \rrlc s}\label{exxx}

\subsection{Preperiodic series}	\label{expreper}
A power series with bounded coefficients, admitting an analytic continuation
outside the unit circle or not, may have no renascent right limit at all.
Think for instance of any non-trivial series $\sum_{k \ge 0} a_k z^k$ with $\lim_{k \to \infty} a_k = 0$.
In the case of the geometric series $\sum_{k\ge 0} z^k$, there is of course only
one right limit, which is renascent, and indeed the corresponding \rrlc\ is
$-\sum_{n=-1}^{-\infty} z^n = (1-z)\ii$,
as expected.
In the case of a preperiodic sequence, 
$g(z) = \sum_{k=0}^\infty a_k z^k$ with $a_k = a_{k+p}$ for
$k\ge m$ (as in Remark ~\ref{finite}), 
it is easy to see 
that the existence of a renascent right limit implies that the
sequence $\{a_k\}_{k=0}^\infty$ is in fact periodic,
with $g(z) = \frac{1}{1-z^p} \sum_{k=0}^{p-1} a_k z^k$:
A periodic sequence has a unique renascent right limit (obtained by periodic
extension from $\Natural=\{0,1,2,\ldots\}$ to~$\integer$), a preperiodic non-periodic sequence
does not have any, but it has a single weakly renascent right limit.

\subsection{Examples with Borel monogenic extensions}\label{Borex}
Holomorphic functions of the form
\begin{equation}	\label{eqdefBWD}
g(z) = \sum_{m=1}^\infty \frac{\rho_m}{z-\la_m}\, , \qquad |z|<1\, ,
\end{equation}
with $\La = \{\la_1,\la_2,\ldots\}$ a dense subset of the unit circle and
$(\rho_m)_{m\ge1} \in \ell^1(\Natural^*,\complex)$,
were considered in \cite{MS} for dynamical reasons,
particularly in the case where $\La$ consists of all roots of unity since the
simplest small divisor problem gives rise to such a situation.
One can show \cite{RRLC} that the function $h(z)$ defined for $|z|>1$ by the same series of
partial fractions is the unique \rrlc\ of~$g$.
Here the unit circle is a strong natural boundary as soon as none of the
``residues''~$\rho_m$'s vanishes,
yet several results on the connection between~$g$ and~$h$ are available:
Under appropriate assumptions on the size of the $\rho_m$'s one can cross the
unit circle through Diophantine points and define a Borel monogenic function (or
even a $C^\infty$-holomorphic function, with nontangential limits for all
derivatives) in a compact subset of~$\complex$ whose intersection with the unit
circle has positive Haar measure \cite[\S 2.4--2.5]{MS}, so that the resulting
function enjoys a certain quasianalyticity property \cite{MS2};
one can also cross the unit circle through any of the ``poles''~$\la_m$ by means of a
quasianalytic generalised Laurent series \cite[\S 4.1]{MS}.

\subsection{Hecke's example}\label{Heckex}
In \cite[(1.14)]{BrSi} one finds Hecke's example
\[
\tilde g(z) = \sum_{k =0}^\infty \{ k \theta\} z^k\, ,
\]
where~$\th$ is a fixed irrational number and $\{\cdot \}$ denotes the fractional part.
Let us consider
\[
g(z) = \sum_{k =0}^\infty \big(\{ k \theta\}-\tfrac{1}{2}\big) z^k
= \tilde g(z)-\frac{1}{2}\frac{1}{1-z}\, ,
\]
so as to deal with the zero mean-value observable $\{\cdot \}-\tfrac{1}{2}$
evaluated along an orbit of the $\th$-rotation.
This is an example of a function with a single renascent right limit for which
the unit circle is a strong natural boundary. 

 Indeed, all the right limits can be easily determined in this case: The
only renascent one is $b_-=\big\{ b_n^- = \{ n
\theta\}-\tfrac{1}{2}\big\}_{n\in\integer}$
and, due to the discontinuity of the fractional part, and there is another right
limit $\{ b_n^+ \}_{n\in\integer}$, with $b_0^+ = \frac{1}{2}$ and $b_n^+=b_n^-$
for $n\in\integer^*$.
By \cite[Thm 3.1]{BrSi} this implies strong natural boundary.
(The fact that the unit circle is a natural boundary can also be deduced from
Weyl's criterion, which shows that
$\lim_{z \NT e^{2\pi i m\theta}} (z-e^{2\pi i m\theta}) g(z) \ne 0$
for every nonzero integer~$m$, so that there is a countable dense
set of pole-like singularities on the unit circle  \cite{KuNi}.)
Observe that the unique \rrlc\ of~$g$ is
\[
g_{b_-}(z) = -\sum_{n =-1}^{-\infty} \big(\{ n \theta\} - \tfrac{1}{2} \big) z^n\, .
\]
Since $\{-x\}-\tfrac{1}{2} = - \big( \{x\} - \tfrac{1}{2} \big)$ if $x \notin
\integer$, we find $g_{b_-}(z) = g(z\ii) + \tfrac{1}{2}$.

\section{Deducing {Theorem~\ref{oldtheorem}} from \cite{Ba} and \cite{BS1}}
\label{gapp}

Observe first that the assumption in \cite{Ba, BS1} that
$\varphi \in C^1$  was only used to give a meaning
to \eqref{7} without integration by parts.

We first show how to obtain \eqref{formula} from the arguments in \cite[Prop.~
4.6]{BS1} (note that this is similar in spirit to the computation in
\cite[\S 17]{Ru3}).
There exists a decomposition
$\rho=\rho_{sal}+\rho_{reg}$  of the invariant
density of $f$ (which is of bounded variation) into 
$\rho_{sal}= \sum_{n = 1}^{\infty} s_n H_{c_n}$, where $H_x$
is a Heaviside function at $x$, and $s_k = f'(c_k) s_{k+1}$,
and $\rho_{reg}'$ is of bounded variation \cite[Prop.~3.3]{Ba}. Let $\LL$
be the usual transfer operator of $f$ acting on $BV$ by
$$
(\LL \psi) (x)=\sum_{f(y)=x} \frac{\psi(y)}{|f'(y)|}\, .
$$
This operator has a simple eigenvalue at $1$, for $\LL \rho=\rho$
and $\int \LL(\psi)\, dx=\int \psi\, dx$, and the rest of the spectrum
lies in a disc of radius $\kappa<1$.

Setting $s_1=-\lim_{x \uparrow c_1}\rho(x)
$, \cite[Prop.~4.4]{Ba} gives
\begin{align}\label{formm}
\Psi_\varphi(z)
&=-\sum_{j=1}^{\infty} \varphi(c_j)  \sum_{k=1}^{j} z^{j-k}  
  \frac {s_1 X(c_k)}{(f^{k-1})'(c_1)}\\
\nonumber &\qquad\qquad\qquad-
\int (1- z\LL)^{-1} (X' \rho_{sal}+(X\rho_{reg})') \varphi \, dx  \, ,
\end{align}
where the  second term is meromorphic in the disc
of radius $\kappa^{-1}$, with at most a single pole at $z=1$, this pole
being simple,
but the residue of this pole vanishes here since
$\int \varphi \, d\mu=0$.
Set
$$
\Psi^{sing}_\varphi(z)=-\sum_{j=1}^{\infty} \varphi(c_j)  \sum_{k=1}^{j} z^{j-k}  
  \frac {s_1 X(c_k)}{(f^{k-1})'(c_1)}=
  \sum_{n=0}^\infty z^n \sum_{k=1}^\infty \frac{s_1 X(c_k)}
  {{(f^{k-1})'(c_1)}}  \varphi(c_{k+n})\, ,
$$
(which is clearly holomorphic in $|z|<1$),
and
\begin{equation}\label{holformula}
\Psi_\varphi^{hol}(z)
=- \int (1- z\LL)^{-1} (X' \rho_{sal}+(X\rho_{reg})') \varphi \, dx  \, .
\end{equation}
To analyse  $\Psi^{sing}(z)$,
we adapt the proof of \cite[Prop.~4.6]{BS1}.
Define for $\ell \ge 1$
$$
\alpha(c_\ell, z)= -\sum_{j= 1}^ \infty \frac{X(f^j(c_\ell))}{z^j (f^ j)'(c_\ell)}
\, .
$$
Clearly,
$z\mapsto \alpha(c_1, z)$
is analytic in $\{ z \in \complex \mid |z| \min |f'|> 1\}$
and we have
\begin{equation}\label{derivative}
\partial_z \alpha(c_1, z)|_{z=1}= \sum_{j=1}^ \infty \frac{j \ X(f^{j}(c_1))}{(f^ j)'(c_1)}
\, .
\end{equation}

For $(\min |f'|)^ {-1} < |z| < 1$,
the proof of \cite[Prop.~4.6]{BS1} gives 
\begin{align*}
\Psi^{sing}(z)=s_1 \Big[   \big( X(c_1) -\alpha(c_1, z)\big) \sum_{j=1}^{\infty} \varphi(c_j) z^{j-1} 
-  \sum_{j=1}^{\infty}  \frac{\varphi(c_j)\ \alpha(c_j,z)}{(f^{j-1})' (c_1)}\Big] \, .
\end{align*}
Finally, we set
\begin{equation}\label{defU}
\UU(z)= s_1    ( X(c_1) -\alpha(c_1, z) ) \, ,
\end{equation}
and
\begin{equation}\label{defV}
\VV_\varphi(z)=-s_1  
\sum_{j=1}^{\infty}  \frac{\varphi(c_j)\ \alpha(c_j,z)}{(f^{j-1})' (c_1)}\, .
\end{equation}

Clearly, $X$ is horizontal if and only if
$\UU(1)=0$. Also,  (\ref{derivative}) implies
that  $X$ is horizontal of order $2$  if and only if $\UU(1) = \UU'(1)=0$.
In fact, for any $1\le \ell \le H-1$, we have the following generalisation
of \eqref{derivative}:
\begin{equation}\label{derivative'}
\partial^{\ell}_z \alpha(c_1, z)|_{z=1}= 
(-1)^{\ell-1}\sum_{j=\ell}^ \infty \frac{j!}{(j-\ell)!} \frac{ X(f^{j}(c_1))}{(f^ j)'(c_1)}
\, .
\end{equation}
Therefore,  $X$ is horizontal of order
$H\ge 1$ 
(recall \eqref{orderH}) if and only if  $\UU(1) = \UU'(1) = \cdots =\UU^{(H-1)}(1)=0$.

Note also that if $X$ is horizontal, we have
(convergence of the first series implies convergence of the second)
\begin{equation}\label{mystery2}
\VV_\varphi(1)=\sum_{j=1}^{\infty} \varphi(c_j)
\sum_{k=j}^\infty \frac {s_1 X(c_{k+1})}{(f^{k})'(c_1)}=-\sum_{j=1}^{\infty} \varphi(c_j)
\sum_{k=1}^j \frac {s_1 X(c_k)}{(f^{k-1})'(c_1)}
\, .
\end{equation}

Finally, if $c$ is not periodic then   $f$ is ``good" in the sense of
\cite{BS1} so that \cite[Thm 5.1, Prop.\ 4.3, Lemma 4.4]{BS1} give \eqref{mystery}.

\section{Birkhoff and Wiener--Wintner typicality}\label{BWW}

\begin{lemma}[Birkhoff typical implies Wiener--Wintner typical]\label{WW}
Let  $X$ be a compact metric space,
let $T:X\to X$ be  continuous, and let $\mu$ be a mixing
invariant probability measure for $T$. 
Let $c\in X$ be so that for any continuous function
$\varphi \col X \to \complex$  (Birkhoff typicality)
$$
\lim_{n \to \infty} \frac{1}{n}\sum_{k=0}^{n-1} \varphi\big(T^k(c)\big)=
\int \varphi \, d\mu .
$$
Then, for any continuous function
$\varphi:X \to \complex$  (Wiener--Wintner typicality)
$$
 \lim_{n \to \infty} \frac{1}{n}\sum_{k=0}^{n-1}
e^{i\omega k} \varphi\big(T^k(c)\big)=0,
\qquad \forall \omega \in (0,2\pi ) .
$$
\end{lemma}

\begin{proof}
Our proof is inspired from \cite[\S2.4]{As}.
We shall use the van der Corput inequality (see e.g.\ \cite{As}): For all
$n\ge 1$,  each sequence $(u_k, 0\le k \le n-1)$ of complex numbers,
and every integer $h$ between $1$ and $n-1$, we have
$$
\abs[\bigg]{ \sum_{k=0}^{n-1} u_k }^2
\le \frac {n+h}{h+1} \sum_{k=0}^{n-1} \abs*{ u_k }^2
+
\frac {2(n+h)}{(h+1)^2} 
\sum_{\ell=1}^h (h+1-\ell)
\abs[\bigg]{ \sum_{k=0} ^{n-\ell -1} \overline{u_{k+\ell}} u_k }.
$$
We first decompose for each  $n$
 $$
 \frac{1}{n}\sum_{k=0}^{n-1}
e^{i \omega k} \varphi\big(T^k(c)\big)=
 \frac{1}{n}\sum_{k=0}^{n-1}
e^{i \omega k} \Big( \varphi\big(T^k(c)\big)-\int \varphi\, d\mu \Big)
+ \frac{\int \varphi\, d\mu}{n} \, \sum_{k=0}^{n-1}
e^{i \omega k} .
$$
For any $\omega\ne0$ we have $\lim_{n \to \infty} \frac{1}{n} \sum_{k=0}^{n-1}
e^{i \omega k}=\lim_{n \to \infty} \frac{1}{n} \frac{1-e^{i n \omega}}{1-e^{i\omega}} 
=0$. We can thus assume $\int \varphi\, d\mu=0$.
Since $\abs*{e^{i\omega k}} = 1$  the van 
der Corput bound implies
\begin{align*}
& \abs[\bigg]{ \frac{1}{n}\sum_{k=0}^{n-1}
e^{i \omega k} \varphi\big(T^k(c)\big) }^2 \le 
\frac {n+h}{n^2(h+1)} \sum_{k=0}^{n-1} \abs[\big]{ \varphi\big(T^k(c)\big)}^2 \\[1ex]
&\qquad\qquad\qquad\qquad
+ \frac {2(n+h)}{n^2(h+1)^2} \sum_{\ell=1}^h (h+1-\ell)
\abs[\Bigg]{ \sum_{k=0} ^{n-\ell -1} \overline{ \varphi\big(T^{k+\ell}(c)\big)} 
  \varphi\big(T^k(c)\big) }.
\end{align*}
Now, on the one hand
$$ \frac {n+h}{n^2(h+1)} \sum_{k=0}^{n-1} \abs[\big]{ \varphi\big(T^k(c)\big) }^2
\le  \frac {2}{n(h+1)} \sum_{k=0}^{n-1} \abs[\big]{ \varphi\big(T^k(c)\big) }^2. 
$$
Since $c$ is Birkhoff typical and $\abs{\varphi}^2$ is continuous,
we have  
$$\lim_{n \to \infty}
\frac {1}{n} \sum_{k=0}^{n-1} \abs[\big]{ \varphi\big(T^k(c)\big) }^2=
 \int \abs{\varphi}^2\, d\mu.
$$ 
Therefore for each
 $\epsilon >0$ and all large enough $n$ (depending on
 $\epsilon$), we have for all $h$ between $1$ and $n-1$
\begin{equation}
\label{vC1} \frac {2}{n(h+1)} \sum_{k=0}^{n-1} \abs[\big]{ \varphi\big(T^k(c)\big) }^2
\le  \frac {2}{h+1}  \Big( \int \abs{\varphi}^2\, d\mu + \epsilon \Big).
\end{equation}
On the other hand
%
\begin{multline*}
%
\frac {2(n+h)}{n^2(h+1)^2} 
\sum_{\ell=1}^h (h+1-\ell)
\abs[\Bigg]{ \sum_{k=0} ^{n-\ell -1} \overline{ \varphi\big(T^{k+\ell}(c)\big)} 
 \varphi\big(T^k(c)\big) } 
\\[1ex]
%
%
\le \frac {4}{n(h+1)} \sum_{\ell=1}^h 
\abs[\Bigg]{ \sum_{k=0} ^{n-\ell -1} \overline{ \varphi\big(T^{k+\ell}(c)\big)} \cdot
 \varphi\big(T^k(c)\big) }.
\end{multline*}
%
Since $\varphi$ is continuous, the function
 $\overline {\varphi\circ T^\ell}\cdot  \varphi$ is continuous for each $\ell \ge 1$.
 Since $c$ is Birkhoff typical, we get for all $\ell\ge 1$
 $$ 
\lim_{n \to \infty}
\frac {1}{n}  
  \sum_{k=0} ^{n-\ell -1} \overline{ \varphi\big(T^{k+\ell}(c)\big)} 
 \varphi\big(T^k(c)\big) =
 \int  \overline{ \varphi\circ T^{\ell}} \cdot 
 \varphi \, d\mu .
 $$
 Thus, for each fixed $\epsilon >0$
and  $h\ge 1$, and all large enough $n>h$  (depending 
on $h$ and $\epsilon$) 
$$
 \frac {4}{n(h+1)} 
\sum_{\ell=1}^h 
\abs[\Bigg]{ \sum_{k=0} ^{n-\ell -1} \overline{ \varphi\big(T^{k+\ell}(c)\big)} \cdot
 \varphi\big(T^k(c)\big)} \le
 \frac{4}{h+1}
 \Bigg( \sum_{\ell=1}^h 
 \abs[\Big]{ \int  \overline{ \varphi\circ T^{\ell}} \cdot 
 \varphi \, d\mu } + \epsilon \Bigg).
 $$
 Since $\mu$ is mixing and
  $\int \varphi\, d\mu=0$, we get
$$
\lim_{\ell \to \infty} \abs[\Big]{ \int \overline{ \varphi\circ T^{\ell}} \cdot 
 \varphi \, d\mu }
= \abs[\Big]{ \int \varphi\, d\mu\cdot \int \bar \varphi\, d\mu } =0,
$$
and therefore
 $$
 \lim_{h \to \infty} \frac{1}{h+1}
 \sum_{\ell=1}^h
 \abs[\Big]{ \int  \overline{ \varphi\circ T^{\ell}} \cdot 
 \varphi \, d\mu } = 0.
 $$
Thus, for each $\epsilon >0$, 
 all large enough $h$  (depending on $\epsilon$), all
 large enough $n>h$  (depending on
$h$ and $\epsilon$) we get
\begin{equation}\label{vC2} \frac {4}{n(h+1)} 
\sum_{\ell=1}^h 
\abs[\Bigg]{ \sum_{k=0} ^{n-\ell -1} \overline{ \varphi\big(T^{k+\ell}(c)\big)} \cdot
 \varphi\big(T^k(c)\big) } \le  8 \epsilon .
 \end{equation}
 Take  $\epsilon$ small, $h(\epsilon)$ large,
$n(h,\epsilon)$ even larger, and use \eqref{vC1} and \eqref{vC2}. 
\end{proof} 

 \bibliographystyle{alpha}

\begin{thebibliography}{10}


\bibitem{Ag} S. Agmon, \textit{Sur les s\'eries de Dirichlet,}
Ann. Sci. \'Ecole Norm. Sup. (3) {\bf 66} (1949) 263--310.

\bibitem{As} I. Assani, \textit{Wiener--Wintner Ergodic Theorems,} World Scientific, Singapore, 2003.


\bibitem{A} I. Assani, Private communication (2012).

\bibitem{Babook} V. Baladi, \textit{Positive Transfer Operators and Decay of Correlations,} v+314 pp, World Scientific Advanced
Series in Nonlinear Dynamics, Vol 16 (2000).

\bibitem{Ba} V. Baladi,
\textit{On the susceptibility function of piecewise expanding interval
maps,} Comm. Math. Phys. {\bf 275} (2007) 839--859.



\bibitem{BS1} V. Baladi and D. Smania,
\textit{Linear response for piecewise expanding unimodal maps,} 
Nonlinearity {\bf 21} (2008) 677--711.

\bibitem{BS2} V. Baladi and D. Smania,
\textit{Linear response for smooth
deformations of  generic nonuniformly hyperbolic unimodal maps,} 
arxiv preprint (2010), to appear in Ann. Sci. \'Ec. Norm. Sup\'er. 

\bibitem{BaN} V. Baladi,
\textit{Linear response despite critical points,} Nonlinearity 
{\bf 21} (2008) T81--T90. 

\bibitem{BBS}
V. Baladi, M. Benedicks, and D. Schnellmann,
\textit{Whitney H\"older continuity of the SRB measure for transversal families
of smooth unimodal maps,} in preparation (2012).

\bibitem{Be} M. Benedicks, Private communication (2011).

\bibitem{BC} L. S. Block and W. A. Coppel, \textit{Dynamics in One Dimension,} Lecture Notes in Mathematics,
{\bf 1513} Springer-Verlag, 1992.

\bibitem{Bor} E. Borel, \textit{Le\c cons sur les fonctions monog\`enes uniformes d'une variable complexe,} Gauthier-Villars, Paris (1917).

\bibitem{BrSi} J. Breuer and B. Simon,
\textit{Natural boundaries and spectral theory,}
Adv. Math. {\bf 226} (2011) 4902--4920.

{\bibitem{Bro} A. Broise, \textit{Transformations dilatantes de l'intervalle et th\'eor\`emes limites,} in  Etudes spectrales d'op\'erateurs de transfert et applications, Ast\'erisque {\bf 238} (1996)  1--109.}

{\bibitem{Br} H. Bruin, Private communication (2012).}

\bibitem{Bu} J. Buzzi, \textit{Specification on the interval,}
Trans. Amer. Math. Soc. {\bf 349} (1997) 2737--2754.

\bibitem{CC} G. Cohen and J.-P. Conze, \textit{The CLT for rotated ergodic sums and related processes,}  Preprint (2012).

\bibitem{FMT} M. Field, I. Melbourne, and A. T\"or\"ok, 
\textit{Decay of correlations, central limit theorems and approximation
by Brownian motion for compact Lie group extensions,}
Ergodic Theory Dynam. Systems {\bf 23}  (2003) 87--110.

\bibitem{Go} S. Gou\"ezel, 
\textit{Almost sure invariance principle for dynamical
systems by spectral methods,} Annals  of Probab.
 {\bf 38} (2010) 1639--1671.


\bibitem{HK} F. Hofbauer and  G. Keller, 
\textit{Ergodic properties of invariant measures for piecewise
monotonic transformations,} Math. Z. {\bf 180} (1982) 119--140.

\bibitem{Ke78}  G. Keller, 
\textit{Piecewise monotonic transformations and exactness,}
Collection: Seminar on Probability, Rennes, 1978; Exp. 6, 32 pp.
Univ. Rennes (1978).

\bibitem{Ko}  Z.S. Kowalski, \textit{Invariant measure for piecewise monotonic
transformations has a lower bound on its support,} Bull Acad. Pol. Sci.
Math. {\bf 27} (1979) 53--57.

\bibitem{KuNi} L. Kuipers and  H. Niederreiter,
\textit{Uniform distribution of sequences,} Wiley, New York, 1974.

\bibitem{Li} C. Liverani,
\textit{Decay of correlations in piecewise expanding maps,} J. Stat. Phys. {\bf 78} (1995) 1111--1129.

\bibitem{MS} S. Marmi and D. Sauzin,
\textit{Quasianalytic monogenic solutions of a cohomological equation,} Mem. Amer. Math. Soc. {\bf 164} (2003) no. 780.

\bibitem{MS2} S. Marmi and D. Sauzin,
\textit{A quasianalyticity property for monogenic solutions of
small divisor problems,}  Bull.  Brazilian Math. Soc., New
Series {\bf 42} (2011) 45--74.

\bibitem{RRLC} S. Marmi, D. Sauzin, and G. Tiozzo,
\textit{Generalised continuation by means of right limits,}
in preparation.

\bibitem{MNJLM} I. Melbourne and M. Nicol, 
\textit{Statistical properties of endomorphisms and compact group
extensions,}
J. London Math. Soc. {\bf 70} (2004) 427--446.

\bibitem{MNSD} I. Melbourne and M. Nicol, 
\textit{Statistical limit laws for equivariant observations,}
Stoch. and Dyn. {\bf 4} (2004) 1--13.



\bibitem{NMA} 
M.
Nicol,  I. Melbourne, and P. Ashwin, \textit{Euclidean extensions of dynamical systems,} Nonlinearity {\bf 14} (2001) 275--300.

\bibitem{RS}
W.T.~Ross and H.~S.~Shapiro, 
\textit{generalised analytic continuation,} 
University Lecture Series, 25. Amer. Math. Soc., Providence, RI,
2002.


\bibitem{Ro} V.A. Rohlin, 
\textit{Exact endomorphisms of a Lebesgue space,} (Russian)
Izv. Akad. Nauk SSSR Ser. Mat. {\bf 25} (1961) 499--530. 

\bibitem{Ru1}
D. Ruelle, 
\textit{Application of hyperbolic dynamics to physics: some problems and conjectures,}  
Bull. Amer. Math. Soc.   {\bf 41}  (2004) 275--278.

\bibitem{Ru2} D. Ruelle, \textit{Differentiating the a.c.i.m. of an interval map with respect to $f$,} Comm. Math. Phys. {\bf 258} (2005) 445--453.

\bibitem{Ru3} D. Ruelle,  \textit{Structure and f-dependence of the A.C.I.M. for a unimodal map f is Misiurewicz type,} Comm. Math. Phys. {\bf 287}  (2009) 1039--1070.

\bibitem{Sc} D. Schnellmann,  \textit{Typical points for one-parameter families of piecewise expanding maps of the interval,}
Discrete Contin. Dyn. Syst. {\bf 31} (2011) 877--911. 

\bibitem{Sh} W. Shen, Private communication (2011).

\bibitem{vBC} H. Van Den Bedem and N. Chernov,
\textit{Expanding maps of an interval with holes,}
Ergodic Theory  Dynam. Systems {\bf 22}  (2002) 637--654.

\bibitem{Wi} R. Wittmann,  \textit{A general law of iterated logarithm,} Z. Wahrsch. Verw. Gebiete {\bf 68} (1985) 521--543.

\end{thebibliography}

\enddocument